\documentclass[12pt]{amsart}

\usepackage[margin=1.1in]{geometry}
\usepackage{microtype}

\usepackage{graphicx}
\usepackage{amssymb}
\usepackage{enumerate}
\usepackage{times}
\usepackage{cite}

\usepackage{xcolor}

\usepackage[colorlinks=true, pdfstartview=FitV, linkcolor=blue, citecolor=blue, urlcolor=blue]{hyperref}

\def\Xint#1{\mathchoice
{\XXint\displaystyle\textstyle{#1}}%
{\XXint\textstyle\scriptstyle{#1}}%
{\XXint\scriptstyle\scriptscriptstyle{#1}}%
{\XXint\scriptscriptstyle\scriptscriptstyle{#1}}%
\!\int}
\def\XXint#1#2#3{{\setbox0=\hbox{$#1{#2#3}{\int}$}
\vcenter{\hbox{$#2#3$}}\kern-.5\wd0}}

\def\avgint{\Xint-}

\newtheorem{theorem}{Theorem}[section]
\newtheorem{lemma}[theorem]{Lemma}
\newtheorem{prop}[theorem]{Proposition}
\newtheorem{corollary}[theorem]{Corollary}

\newtheorem{definition}[theorem]{Definition}
\newtheorem{example}[theorem]{Example}

\theoremstyle{definition}

\theoremstyle{remark}
\newtheorem{remark}[theorem]{Remark}

\numberwithin{equation}{section}
 \allowdisplaybreaks

\newcommand{\R}{\mathbb R}

\newcommand{\BB}{\mathbf B}

\newcommand{\Q}{\mathcal Q}
\newcommand{\D}{\mathcal D}
\newcommand{\Ss}{\mathcal S}

\newcommand{\F}{\mathcal F}

\newcommand{\W}{\mathcal{W}}

\newcommand{\A}{\mathcal A}
\newcommand{\B}{\mathcal B}

\newcommand{\Zz}{\mathcal Z}

\DeclareMathOperator*{\esssup}{ess\,sup}
\DeclareMathOperator{\sgn}{sgn}

\DeclareMathOperator{\op}{{op}}
\DeclareMathOperator{\diag}{diag}

\DeclareMathOperator{\conv}{conv}
\DeclareMathOperator{\clconv}{\overline{conv}}
\DeclareMathOperator{\cl}{cl}

\newcommand{\K}{\mathcal{K}}
\newcommand{\cs}{\K_{cs}(\R^d)}

\newcommand{\bcs}{\K_{bcs}(\R^d)}

\newcommand{\Rdf}{\mathcal{R}}

\newcommand{\wM}{\widetilde{M}}

\newcommand{\Md}{\mathcal{M}_d}
\newcommand{\Sd}{\mathcal{S}_d}

\newcommand{\loc}{\text{loc}}
\title{Off-diagonal matrix extrapolation for Muckenhoupt bases}

\author{David Cruz-Uribe OFS}
\address{Department of Mathematics \\ University of Alabama \\
Tuscaloosa, AL 35487, USA}
\email{dcruzuribe@ua.edu}

\author{Fatih \c{S}irin}
\address{Department of Mathematics \\ Hali\c{c} University \\ 5. Levent Campus, Istanbul, 34060, T\"urkiye}
\email{fatihsirin@halic.edu.tr}

\thanks{The first author is partially supported by a Simons Foundation Travel Support for Mathematicians Grant and by NSF grant DMS-2349550.  The second author is supported by 2219 - International Postdoctoral Research Fellowship Program for Turkish Citizens from T\"ubitak, the Scientific and Technological Research Council of T\"urkiye.}



\begin{document}

\begin{abstract}
In this paper we extend the theory of Rubio de Francia extrapolation for matrix weights, recently introduced by Bownik and the first author~\cite{bownik-cruz-uribe}, to off-diagonal extrapolation.  We also show that the theory of matrix weighted extrapolation can be extended to matrix $\A_p$ classes defined with respect to a general basis, provided that a version of the Christ-Goldberg maximal operator is assumed to be bounded.  Finally, we extend a recent result by Vuorinen~\cite{VUORINEN2024109847} and show that all of the multiparameter bases have this property.
\end{abstract}

\maketitle
%%% start Writing Paper Here %%%

\section{Introduction}

In this paper we further develop the theory of Rubio de Francia extrapolation for matrix weights.  To put our results in context, we briefly sketch the theory of matrix weights and extrapolation.  For a more complete history, see~\cite{bownik-cruz-uribe,cruzuribe2024} and the references they contain.  A matrix weight $W$ is a $d\times d$ self-adjoint matrix function whose coefficients are measurable functions that are finite almost everywhere, and such that $W(x)$ is invertible for almost every $x$. For brevity we will say that $W$ is finite and invertible almost everywhere.   For $1\leq p<\infty$ we define the space $L^p(\R^n, W)$ to be the Banach space of measurable, vector-valued functions $f :\R^n \rightarrow \R^d $  with the norm
\[ \|f\|_{L^p(\R^n,W)} = \|Wf\|_{L^p(\R^n,\R^d)} = \bigg(\int_{\R^n} |W(x)f(x)|^p\,dx \bigg)^{\frac{1}{p}} < \infty.  \]
Traditionally, this norm is defined with $W$ replaced by $W^{\frac{1}{p}}$, so that when $d=1$ (i.e., in the scalar case) the integral becomes
\[ \int_{\R^n} |f(x)|^p w(x)\,dx.  \]
However, for the theory of extrapolation it is more convenient to use the above formulation; in particular, the first equality still make sense for $p=\infty$, and this definition allows a unified treatment of endpoint cases.  Moreover, this formulation is more natural for off-diagonal inequalities, even in the scalar case.

Matrix weighted norm inequalities for singular integrals have been considered by a number of authors, beginning with Nazarov, Treil and Volberg~\cite{MR1428988,MR1428818,Vol}.  (See~\cite{cruzuribe2024} for additional references.)  This led to the development of the theory of matrix $A_p$ weights, generalizing the classical scalar theory of Muckenhoupt $A_p$ weights.  In~\cite{bownik-cruz-uribe}, Bownik and the first author resolved a long open problem in the theory by extending the theory of Rubio de Francia extrapolation to the setting of matrix weights.  To state their result, we define the relevant weight classes.
For $1<p<\infty$, we say a matrix weight $W$ is in $\A_p$ if
\[ [W]_{\A_p} = \sup_Q \bigg(\avgint_Q \bigg( \avgint_Q
    |W(x)W^{-1}(y)|_{\op}^{p'}\,dy\bigg)^{\frac{p}{p'}}\,dx\bigg)^{\frac{1}{p}} <
    \infty. \]
  Here, the supremum is taken over all cubes with sides parallel to the coordinate axes.
  We also define the classes $\A_1$ and $\A_\infty$ by
  \[ [W]_{\A_1} = \sup_Q \esssup_{x\in Q} \avgint_Q
    |W^{-1}(x)W(y)|_{\op}\,dy < \infty,  \]
and
\[ [W]_{\A_\infty} =\sup_Q \esssup_{x\in Q} \avgint_Q
  |W(x)W^{-1}(y)|_{\op}\,dy < \infty.  \]

\begin{theorem} \label{thm:matrix-extrapolation-intro} Let $T$ be a
  sublinear operator.  Suppose that for some $p_0$,
  $1 \leq p_0\leq\infty$, there exists an increasing function
  $N_{p_0}$ such that for every $W_0\in \A_{p_0}$,
\begin{equation}\label{ext1}
 \|Tf\|_{L^{p_0}(\R^n, W_0)}
  \leq N_{p_0}([W_0]_{\A_{p_0}})\|f\|_{L^{p_0}(\R^n, W_0)}. 
  \end{equation}
  Then for all $p$, $1<p<\infty$, and 
  for all $W\in \A_p$,
\begin{equation}\label{ext2}
    \|Tf\|_{L^p(\R^n,W)}
  \leq N_{p}(p,p_0,n,d,[W]_{\A_{p}})\|f\|_{L^{p}(\R^n,W)},
  \end{equation}
  where
  \[ N_p (p,p_0,n,d,[W]_{\A_{p}})
    = C(p,p_0)N_{p_0}\bigg(C(n,d,p,p_0)[W]_{\A_p}^{\max\big\{\frac{p}{p_0},\frac{p'}{p_0'}\big\}}\bigg). \]
\end{theorem}

The goal of this paper is to generalize Theorem~\ref{thm:matrix-extrapolation-intro} in two ways.  First, we prove a matrix version of off-diagonal extrapolation.  In the scalar case this was first proved by Harboure, Mac\'\i as and Segovia~\cite{harboure-macias-segovia88} for the class of  $A_{p,q}$ weights introduced by Muckenhoupt and Wheeden~\cite{muckenhoupt-wheeden74}.  This class was generalized to the matrix setting by Isralowitz and Moen~\cite{MR4030471} who used it to prove matrix weighted inequalities for the fractional integral operator.  Given $p,\,q$,  $1<p\le q<\infty$, we say that $W\in \A_{p,q}$ if
\[
[W]_{A_{p,q}} = \sup_{Q} \left( \avgint_Q \left( \avgint_{Q} |W(x) W^{-1}(y)|^{p'}_{\op} \,dy \right)^{\frac{q}{p'}} \,dx \right)^{\frac{1}{q}}  < \infty.
\]
When $p=1$ and $q<\infty$, 
we say that  $W \in \A_{1,q}$ if
\[
[W]_{\A_{1, q}} = \sup_{Q} \esssup_{x \in Q} \left( \avgint_Q  |W^{-1}(x) W(y)|_{\op}^q \,dy \right)^{\frac1q} < \infty.
\]
Finally, when $p>1$ and $q=\infty$ we say that $W\in \A_{p,\infty}$ if
\[
[W]_{\A_{p, \infty}} = \sup_{Q} \esssup_{x \in Q} \left( \avgint_Q  |W(x) W^{-1}(y)|_{\op}^{p'} \,dy \right)^{\frac{1}{p'}} < \infty.
\]
Note that if $p=q$, these reduce to the $\A_p$ classes defined above.

In the scalar case, the proof of off-diagonal extrapolation relies heavily on the fact that $w\in \A_{p,q}$ if and only if $w^s\in A_r$, where $\frac{1}{p} - \frac{1}{q} =\frac{1}{s'}$  and $r=\frac{q}{s}$.  (See Section~\ref{section:muckenhoupt-basis} below.)    However, in the matrix case the  assumption that $W\in \A_{p,q}$  is strictly weaker than the  assumption that $W^s\in \A_r$;  see Example~\ref{example:no-reverse-cordes} below.    And it is this stronger condition that we need  to prove extrapolation.  

\begin{theorem} \label{thm:off-diag-extrapol-cubes}
Let $T$ be a sublinear operator.   Fix $s$, $1\leq s<\infty$.  Suppose that for some $ p_0, q_0, 1 \leq p_0 \leq q_0 \leq \infty $, where $\frac{1}{p_0}-
\frac{1}{q_0}=\frac{1}{s'}$ and $r_0=\frac{q_0}{s}$, there exists a positive increasing function $ N_{p_0,q_0} $ such that for every matrix weight $W_0$ with $ W_0^s \in \A_{r_0} $,
\begin{equation} \label{eqn:odec1}
    \left( \int_{\R^n} \left| W_0(x) Tf(x) \right|^{q_0} \,dx \right)^{\frac{1}{q_0}} 
\leq N_{p_0, q_0} \left( [W_0^s]_{\A_{r_0}}\right)
\left( \int_{\R^n} \left| W_0(x) f(x) \right|^{p_0} \,dx \right)^{\frac{1}{p_0}}.
\end{equation}
Then for all $ p,\, q $ such that $ 1 < p \leq q < \infty $, $ \frac{1}{p} - \frac{1}{q} = \frac{1}{p_0} - \frac{1}{q_0}$, $r=\frac{q}{s}$, and for all matrix weights $W$ with $ W^s \in \A_{r} $,
\begin{equation} \label{eqn:odec2}
\left( \int_{\R^n} \left| W(x) Tf(x) \right|^q \,dx \right)^{\frac{1}{q}} 
\leq N_{p, q}(p_0,q_0,p,q,[W^s]_{\A_{r}})
\left( \int_{\R^n} \left| W(x) f(x) \right|^p \,dx \right)^{\frac{1}{p}},
\end{equation}
where 
\[
N_{p, q}(p_0,q_0,p,q,[W^s]_{\A_{r}}) = C(d,p_0,q_0,p,q)
N_{p_0, q_0}\big(C(d,p,q,p_0, q_0) [W^s]_{\A_{r}}^{\frac{r}{r_0} + \frac{r'}{r'_0}}\big).
\]
\end{theorem}

\begin{remark}
    When $s=1$, i.e., when $p_0=q_0$, this reduces to a qualitative version of Theorem~\ref{thm:matrix-extrapolation-intro}.  However, we do not recapture the sharp constant in this result.  Nor do we recapture the sharp constant in scalar, off-diagonal extrapolation, given by Lacey, Moen, P\'erez, and Torres~\cite{MR2652182}.  
\end{remark}

\begin{remark}
    It is immediate that Theorem~\ref{thm:off-diag-extrapol-cubes} remains true if we assume that \eqref{eqn:odec1} holds with the weaker assumption that $W \in A_{p_0,q_0}$.  It is an open question whether we can prove \eqref{eqn:odec2} holds if we only assume that $W\in \A_{p,q}$.  Since the $\A_{p,q}$ condition is sufficient for fractional integral operators and the fractional Christ-Goldberg maximal operator to be bounded (see~\cite{MR4030471}), it is not unreasonable to conjecture that this is the case.  However, a very different proof would be required.
\end{remark}

\medskip

Our second generalization is to extend the theory of matrix extrapolation to weight classes defined with respect to a basis. Here, by a basis $\B$ we mean a collection of open sets in $\R^n$--we do not, {\em a priori}, assume any of the conditions that often appear in differentiation theorem.  (E.g. see de Guzm\'an~\cite{MR0457661}.)  In the scalar case, Muckenhoupt $A_p$ weights are defined with respect to averages over sets in the collection $\Q$ of (open) cubes in $\R^n$ with sides parallel to the coordinate axes, since the Hardy-Littlewood maximal operator is defined as the supremum over averages on  cubes.  But it is possible to define maximal operators, and so $\A_p$ weights, by averaging over other sets, such as:
\begin{itemize}
    \item the set $\D$ of dyadic cubes;
    \item the set $\Rdf$ of rectangles with sides parallel to the coordinate axes (e.g., the strong maximal operator);
    \item more generally, the multiparameter basis $\Rdf^\alpha$, $\alpha=(a_1,\ldots,a_j)$, consisting of sets of the form $Q_1\times\cdots \times Q_j$, where $Q_i$ is a cube in $\R^{a_i}$;
    \item the Zygmund basis $\Zz$ of rectangles in $\R^3$ with sides parallel to the coordinate axes and whose sidelengths satisfy $(s,t,st)$ for some $s,\,t>0$. 
\end{itemize} 
Extrapolation was originally extended piecemeal to weighted norm inequalities for $\A_p$ weights defined with respect to specific bases: extrapolation for the basis of dyadic cubes is implicit in the original theory; for the basis of rectangles, see Garc\'\i a-Cuerva and Rubio de Francia~\cite{garcia-cuerva-rubiodefrancia85}; for the Zygmund basis, see Fefferman and Pipher~\cite{MR1439553}.  In~\cite{MR2797562}, it was shown that the theory of extrapolation could be developed for a general basis $\B$ if it was assumed that the basis was a Muckenhoupt basis:  that is, for  every $1<p<\infty$, the maximal operator $M_\B$, defined with respect to averages over sets in the basis $\B$, is assumed to be bounded on $L^p(w)$ provided that $w\in \A_{p,\B}$, the Muckenhoupt class defined with respect to the basis $\B$.  Muckenhoupt bases were introduced in~\cite{perez91} but were implicit in~\cite{jawerth86}.

We extend this approach to matrix weights. Given a basis $\B$ and a matrix weight $W$, we define the Christ-Goldberg maximal operator $M_{W,\B}$ by
\[ M_{W,\B}f(x) = \sup_{B\in \B} \avgint_B |W(x)W^{-1}(y)f(y)|\,dy \cdot \chi_B(x). \]
Note that if there exists $x\in \R^n \setminus \bigcup_{B\in \B} B$, then $M_{W,\B}f(x)=0$.  
We define the corresponding class of $\A_{p,\B}$ weights
by replacing cubes with sets $B\in \B$ in the above definitions: e.g., for $1<p<\infty$,
\[ [W]_{\A_{p,\B}} = \sup_{B\in \B} \bigg(\avgint_B \bigg( \avgint_B
    |W(x)W^{-1}(y)|_{\op}^{p'}\,dy\bigg)^{\frac{p}{p'}}\,dx\bigg)^{\frac{1}{p}} <
    \infty. \]
Note that for both $M_{W,\B}$ and $\A_{p,\B}$ to be well defined, we  only have to assume that $W$ is finite and invertible for almost every $x\in \bigcup_{B\in \B} B$. 

\begin{definition} \label{defn:matrix-muckenhoupt}
    Given a basis $\B$, we say that it is a matrix Muckenhoupt basis if for every $1<p<\infty$ and matrix weight $W\in \A_{p,\B}$, $M_{W,\B} : L^p(\R^n,\R^d) \rightarrow L^p(\R^n)$, and there exists an increasing function $K_p$ such that 
    \[\bigg(\int_{\R^n}|W(x)M_{W,\B}f(x)|^p\,dx \bigg)^{\frac{1}{p}}
    \leq 
    K_p([W]_{\A_{p,\B}})\bigg(\int_{\R^n}|W(x)f(x)|^p\,dx\bigg)^{\frac{1}{p}}.
    \]
\end{definition}

\begin{remark}
    The function $K_p$ could also depend, implicitly, on $d,\,n,\,p$.  Below we are primarily concerned with its dependence on $[W]_{\A_{p,\B}}$.
\end{remark}

\begin{remark}
    As in the scalar case, we make no assumption in Definition~\ref{defn:matrix-muckenhoupt} about the behavior of $M_{W,\B}$ when $p=1$.  Strong-type inequalities generally do not hold when $p=1$, and  weak-type inequalities for matrix weights are more delicate than in the scalar case:  see~\cite{DCU-JI-KM-SP-IRR,DCU-BS-2023}.  In addition, it is known that important bases, such as $\Rdf$, do not satisfy even unweighted weak $(1,1)$ inequalities.  See~\cite{garcia-cuerva-rubiodefrancia85}.  We also do not assume anything when $p=\infty$. Since in this case we have that given any basis $\B$, $M_{W,\B}$ is always bounded on $L^\infty$ if $W\in \A_{\infty,\B}$.  This was proved in~\cite[Proposition~6.8]{bownik-cruz-uribe} for the basis of cubes, but the same proof works for any basis.  Also see the proof of Theorem~\ref{theo_boundednes} below.
\end{remark}

\begin{remark}
    Implicit in Definition~\ref{defn:matrix-muckenhoupt} is that it holds for all $d\geq 1$, so  a matrix Muckenhoupt basis is a scalar Muckenhoupt basis.  It is unknown if the converse is true, or if there is a scalar Muckenhoupt basis that is not a matrix Muckenhoupt basis.
\end{remark}

\begin{remark} \label{remark:sharp-constant-cubes}
    When $\B$ is the basis of cubes $\Q$,  we have that $K_p([W]_{\A_{p,\Q}})=C(n,d,p)[W]_{\A_{p,\Q}}^{p'}$.  This was proved by Isralowitz and Moen~\cite[Theorem~1.3]{MR4030471}.
\end{remark}

Our second result generalizes Theorem~\ref{thm:off-diag-extrapol-cubes} to matrix Muckenhoupt bases.  By Remark~\ref{remark:sharp-constant-cubes}, Theorem~\ref{thm:off-diag-extrapol-cubes} is an immediate consequence.

\begin{theorem} \label{thm:off-diag-extrapol}
Let $T$ be a sublinear operator, and let $\B$ be a matrix Muckenhoupt basis.  Fix $s$, $1\leq s<\infty$. Suppose that for some $ p_0, q_0, 1 \leq p_0 \leq q_0 \leq \infty $, where $\frac{1}{p_0}-\frac{1}{q_0}= \frac{1}{s'}$ and $r_0=\frac{q_0}{s}$, there exists a positive increasing function $ N_{p_0,q_0}$ such that for every matrix weight $W$ with $ W_0^s \in \A_{r_0,\B} $,
\begin{equation*}
    \left( \int_{\R^n} \left| W_0(x) Tf(x) \right|^{q_0} \,dx \right)^{\frac{1}{q_0}} 
\leq N_{p_0, q_0} \left( [W_0^s]_{\A_{r_0,\B}}\right)
\left( \int_{\R^n} \left| W_0(x) f(x) \right|^{p_0} \,dx \right)^{\frac{1}{p_0}}.
\end{equation*}
Then for all $ p,\, q $ such that $ 1 < p \leq q < \infty $, $ \frac{1}{p} - \frac{1}{q} = \frac{1}{p_0} - \frac{1}{q_0}$, $r=\frac{q}{s}$, and for all matrix weights $W$ with $ W^s \in \A_{r,\B} $,
\[
\left( \int_{\R^n} \left| W(x) Tf(x) \right|^q \,dx \right)^{\frac{1}{q}} 
\leq N_{p, q}(\B,p_0,q_0,p,q,[W^s]_{\A_{r,\B}})
\left( \int_{\R^n} \left| W(x) f(x) \right|^p \,dx \right)^{\frac{1}{p}},
\]
where 
\begin{multline*}
    N_{p,q}(\B,p_0,q_0,p,q, [W^s]_{\A_{r,\B}}) \\
    =
C(d,p_0,q_0,p,q) N_{p_0,q_0}\big( C(d,p_0,q_0) K_{r'}([W^{-s}]_{\A_{r',\B}})^{\frac{1}{r_0}} 
K_r([W^s]_{\A_{r,\B}})^{\frac{1}{r_0'}}\big). 
\end{multline*} 
\end{theorem}

\medskip

We next consider the properties of matrix Muckenhoupt bases by considering the relationship between the boundedness of the maximal operator $M_{W,\B}$ and the $\A_{p,\B}$ condition.  It is well-known that in the scalar case, for the basis of cubes, the Muckenhoupt $\A_p$ condition is necessary for the Hardy-Littlewood maximal operator to be bounded on $L^p(w)$, $1<p\leq \infty$.  (The case when $p=\infty$ is not as well-known, but can be found in Muckenhoupt's original paper~\cite{muckenhoupt72}.)  However, more can be said, as was shown by Garc\'\i a-Cuerva and Rubio de Francia~\cite[Section~IV.1]{garcia-cuerva-rubiodefrancia85}.  Let $w$ be a measurable function taking values in $[0,\infty]$; we explicitly do not assume that $0<w(x)<\infty$ a.e.  Then the space $L^p(w)$ is well-defined, using the conventions $0\cdot\infty = 0$ and $t\cdot \infty=\infty$ for $t>0$.  If the Hardy-Littlewood maximal operator is bounded on $L^p(w)$, then {\em as a consequence} one has that $0<w(x)<\infty$ a.e. and the $\A_p$ condition holds.  A similar result holds for a general basis:  see~\cite[Section~3.1]{MR2797562}.

We prove an analogous result for matrix weights.  Given a basis $\B$, let 
\[ \Omega = \bigcup_{B\in \B} B.  \]
We define an equivalence relation on $\B$ by saying that given $B,\,B'\in \B$, then $B\approx B'$ if there exists a finite sequence $\{B_i\}_{i=0}^k$, $B_i\in \B$, $B_0=B$, $B_k=B'$, and for $1\leq i \leq k$, $B_{i-1}\cap B_i \neq \emptyset$.  Since the sets $B_i$ are open, each intersection has positive measure.  Denote the equivalence classes by $\B_j$ and the union of sets in $\B_j$ by $\Omega_j$.  

Now let $W$ be a $d\times d$ matrix function defined on $\Omega$ such that the components of $W$ take values in $[-\infty,\infty]$.  (Without loss of generality, we may assume $W$ is zero on $\R^n\setminus \Omega$.)     We also assume that $W$ is self-adjoint and positive semi-definite, in the sense that $\langle W(x)v,v\rangle \geq 0$ for every $v\in \R^d$ and almost every $x\in \Omega$, using the conventions that $0\cdot \infty = 0$, and $s\cdot \infty+t\cdot\infty= \sgn(s+t)\infty$.

Let $f=(f_1,\ldots,f_d)$, where each $f_i$ is a measurable function taking values in $[-\infty,\infty]$. Define the product $W(x)f(x)$ again using the above conventions.   Let $L^p(\R^n,W)$ consist of all such $f$ that satisfy
\[ \int_{\R^n} |W(x)f(x)|^p\,dx < \infty.  \]
For such a matrix function $W$, define the  maximal operator $\wM_{W,\B}$ on $L^p(\R^n,W)$ by
\[  \wM_{W,\B}f(x) = \sup_{B\in \B} \avgint_B |W(x)f(y)|\,dy \cdot \chi_B(x).  \]
Given a basis $\B$, $W$ as above, and $1<p<\infty$, we say that $\B$ has the weak matrix Muckenhoupt property for $W$ and $p$ if $\wM_{W,\B} : L^p(\R^n, W)\rightarrow L^p(\R^n)$.  Note that if $W$ is finite and invertible almost everywhere, then by the change of variables $f\mapsto W^{-1}f$, we see that this is equivalent to $M_{W,\B} : L^p(\R^n,\R^d)\rightarrow L^p(\R^n)$.   Moreover, we have that $W$ being finite and invertible is (essentially) a  consequence of the weak matrix Muckenhoupt property.

\begin{theorem} \label{thm:weak-muckenhoupt}
    Given a basis $\B$, a matrix function $W$ as above, and $1<p<\infty$, suppose $\B$ has the weak matrix Muckenhoupt property for $W$ and $p$.  Then for each $j$:
    \begin{enumerate}
    \item $W$ is non-trivial on $\Omega_j$:  that is, $0<|W(x)v|<\infty$ for every $v\in \R^d$, $v\neq 0$, and almost every $x\in \Omega$;

    \item $W$ is trivial on $\Omega_j$:  there exists a vector $v\in \R^d$, $v\neq 0$, such that $|W(x)v|=0$ almost everywhere, or $|W(x)v|=\infty$ almost everywhere.
    \end{enumerate}
    Let $\B^*$ be the union of all $\B_j$ on which $W$ is non trivial and let 
    \[ \Omega^* = \bigcup_{B\in \B^*} B.  \]
Then $W$ is finite and  invertible almost everywhere on $\Omega^*$ and in fact, $|W|_{\op} \in L^{p}_\loc(\Omega^*)$ and $|W^{-1}|_{\op} \in L^{p'}_\loc(\Omega^*)$.  Moreover, $W \in \A_{p,\B^*}$:
    \[ \sup_{B\in \B^*} \bigg(\avgint_B \bigg(\avgint_B |W(x)W^{-1}(y)|_{\op}^{p'}\,dy\bigg)^{\frac{p}{p'}}\,dx\bigg)^{\frac{1}{p}} < \infty. \]
\end{theorem} 

Given Theorem~\ref{thm:weak-muckenhoupt}, we could modify our definition of matrix weights as follows. We  say that a matrix function $W$ is a matrix weight on a set $\Omega$ if $W$ is finite and invertible almost everywhere on $\Omega$.   Given a basis $\B$, we consider matrix functions $W$ that are non-trivial on at least one equivalence class $\B_j$ of $\B$.   We then replace $\B$ by $\B^*$ and $W$ by $W^*=W\chi_{\Omega*}$.  If a matrix $W$ is such that $\B$ has the weak matrix Muckenhoupt property for $W$ and some $p$, $1<p<\infty$, then $W^*$ is a matrix weight on $\Omega^*$ and $W^*\in \A_{p,\B^*}$. We define the class $\A_{p,\B}$ to be the collection of all such matrices $W^*$, with the implicit modification of $\B$ into $\B^*$.  We  define the $\A_{p,q,\B}$ classes similarly.  All of our results for general bases then go through with essentially no changes.  This modification would be useful when working with a basis (such as $\D$) that has more than one equivalence class.  To simplify our presentation, however, we will not consider this more general definition, and leave the details to the interested reader.  (Also see Remark~\ref{remark:trivial-matrix} after the proof of Theorem~\ref{thm:weak-muckenhoupt}.)

\medskip

Finally, we give a non-trivial example of matrix Muckenhoupt bases.  
Clearly, $\Q$ and $\D$ are matrix Muckenhoupt bases.  Recently, Vuorinen~\cite[Theorem~1.3]{VUORINEN2024109847}, building off of results proved by Domelevo, {\em et al.}~\cite{domelevo2024boundedness}, proved that the biparameter basis $\Rdf^\alpha$, $\alpha=(a_1,a_2)$, is  a matrix Muckenhoupt basis.  We extend his result to arbitrary multiparameter bases.

\begin{theorem}\label{theo_boundednes}
Fix a multiparameter basis $\Rdf^\alpha$, $\alpha=(a_1,\ldots,a_j)$.  Fix \( p \in (1, \infty] \).  Then for every  \( W \in \A_{p, \mathcal{R}^\alpha} \), 
\[
\|M_{W,\Rdf^\alpha} f\|_{L^p(\R^n)} \leq C(n, d, p) [W]_{\A_{p, \mathcal{R}^\alpha}}^{jp'} \|f\|_{L^p(\R^n; \R^d)}.
\]
Moreover, when $p=\infty$, this inequality holds with \([W]_{\A_{\infty, \mathcal{R}^\alpha}}\) instead of \([W]_{\A_{\infty, \mathcal{R}^\alpha}}^j\).
\end{theorem}

\begin{remark}
    An interesting open problem is to prove that the the Zygmund basis $\Zz$ is a matrix Muckenhoupt basis.
\end{remark}

\medskip

The remainder of this paper is organized as follows.  In Section~\ref{section:muckenhoupt-basis} we state some basic properties of matrices and matrix weights, and then describe the relationship between the   class $\A_{p,q,\B}$ and the class of matrix weights $W$ such that $W^s\in \A_{r,\B}$.  We adapt an example due to Kakaroumpas, Nguyen, and Vardakis~\cite{KNV} to show that the $\A_{p,q,\B}$ condition is strictly weaker.      In Section~\ref{section:tech-lemmas} we gather a number of results about convex set-valued operators needed for the proof of extrapolation.  All of these results were proved for the basis of cubes in~\cite{bownik-cruz-uribe}; here we outline the changes necessary for them to work for general bases.  In Section~\ref{section:main-thm},we prove Theorem~\ref{thm:off-diag-extrapol}. In fact, we will prove a more general theorem, stated in terms of a family of extrapolation pairs.  This abstract approach to extrapolation first appeared in~\cite{cruz-uribe-martell-perez04} and further developed in~\cite{MR2797562}.  The application of this approach to matrix weights is discussed in~\cite{bownik-cruz-uribe}.    In Section~\ref{section:weak-Muckenhoupt} we prove Theorem~\ref{thm:weak-muckenhoupt}. Finally, in Section~\ref{section:multiparameter} we prove Theorem~\ref{theo_boundednes}. 

Throughout this paper, we generally follow the notational conventions established in~\cite{bownik-cruz-uribe}.
 In Euclidean space the constant $n$ will denote
the dimension of $\R^n$, which will be the domain of our functions.  The value
$d$ will denote the dimension of vector and set-valued functions.     For $1\leq p \leq \infty$, $L^p(\R^n)$ will denote the Lebesgue space
 of scalar functions, and $L^p(\R^n,\R^d)$ will denote the Lebesgue
 space of vector-valued functions.  

 Given $v=(v_1,\ldots,v_d)^t \in \R^d$, the Euclidean norm of $v$ will
 be denoted by $|v|$.  The standard orthonormal basis in $\R^d$ will
 be denoted by $\{e_i\}_{i=1}^d$.  The open unit ball 
 $\{ v \in \R^d : |v|<1\}$ will be denoted by ${\mathbf B}$ and its
 closure by $\overline{\mathbf{B}}$.  Matrices will be $d\times d$
 matrices with real-valued entries unless otherwise specified.  The
 set of all such matrices will be denoted by $\Md$.  The set of all
 $d\times d$, symmetric (i.e., self-adjoint), positive semidefinite matrices will be denoted by
 $\Sd$.  We will denote the transpose of a matrix $W$ by $W^*$.  Given
 two quantities $A$ and $B$, we will write $A \lesssim B$, or
 $B\gtrsim A$ if there is a constant $c>0$ such that $A\leq cB$.  If
 $A\lesssim B$ and $B\lesssim A$, we will write $A\approx B$.

\section{Matrix Muckenhoupt weights}
\label{section:muckenhoupt-basis}

In this section we first give some basic results about the operator norms of matrices, define the matrix weight classes we are interested in, and then examine the relationship between the $\A_{p,q,\B}$ weights and matrix weights $W$ such that $W^s\in \A_{r,\B}$.  We conclude with an example to show that the $\A_{p,q,\B}$ is strictly weaker for the basis of cubes.

\subsection*{Matrix operator norms}

Given a matrix $W\in \Md$, its operator norm is defined by 
\[ |W|_{\op} =  \sup_{\substack{v\in \R^d\\ |v|=1}} |Wv|.\]
The following lemmas are very useful for estimating operator norms.

\begin{lemma}{\cite[Lemma 3.2]{MR1928089}}\label{opNorm:equiv}
If $\{v_1, \ldots, v_d\}$ is any orthonormal basis in $\R^d$, then for any  matrix $W\in \Md$ and $r>0$, 
\[
	 \bigg(\sum_{i=1}^d |W v_i|^r\bigg)^{\frac{1}{r}} \approx |W|_{\op}, 
\]
where the implicit constants depend only on $d$ and $r$.
\end{lemma}
\begin{lemma}\label{SelfAdjointCommutes}
Given $U,\,V \in \Ss_d$, $|UV|_{\op} = |VU|_{\op}$.
\end{lemma}

As a corollary we have the following inequality:  let  \( A, B, C, D \in \Ss_d \) be such that
\( |Av| \lesssim |Cv| \) and \( |Bv| \lesssim |Dv| \) for all \( v \in \R^d \).  If we  apply Lemma~\ref{opNorm:equiv} twice, we get
\begin{equation}\label{opNorm-twomatrices}
    |AB|_{\op} \lesssim |CB|_{\op} = |BC|_{\op} \lesssim |DC|_{\op} = |CD|_{\op}.
\end{equation}

The following lemma is referred to as Cordes inequality~\cite[Lemma~5.1, p.~24]{MR890743}.

\begin{lemma} \label{lemma:cordes}
Given matrices $U,\,V \in \Sd$ and $0\leq s \leq 1$, $|U^sV^s|_{\op}\leq |UV|_{\op}^s$.
\end{lemma}

It is well-known that a self-adjoint, positive semidefinite matrix can be diagonalized and the diagonal entries, its eigenvalues, are non-negative.  The same is true for matrix weight; moreover, the diagonalization can be done using measurable matrices; for a proof of the following result, see~\cite[Lemma~2.3.5]{MR1350650}.

\begin{lemma} \label{lemma:diagonal}
    Given a measurable matrix function $W$, there exists a measurable matrix function $U : \R^n \rightarrow \Md$ such that for almost every $x$, $U(x)$ is orthogonal and $U(x)W(x)U^t(x)$ is diagonal; denote it by $\diag(\lambda_1,\ldots,\lambda_d)$.  The non-negative functions $\lambda_i$ are the eigenvalues of $W$;
    moreover, if we denote the columns of $U$ by $u_i$, then each $u_i$ is  measurable and is the eigenfunction associated with $\lambda_i$.
\end{lemma}

It follows from Lemma~\ref{lemma:diagonal} that
    \[ |W(x)|_{\op} = \max(\lambda_1(x),\ldots,\lambda_d(x)).  \]
We can also use it to define the powers of a matrix weight $W$:  for $r>0$, define $W^r$ by
\[ U^* W^r U  = \diag(\lambda_1^r,\ldots,\lambda_d^r). \]

\subsection*{Matrix $\A_p$ weights}
Given a basis $\B$, we define the matrix $\A_{p,\B}$ and $\A_{p,q,\B}$ classes.  For $1<p<\infty$,  a matrix weight $W$ is in $\A_{p,\B}$ if 
\[ [W]_{A_{p,\B}} = \sup_{B\in \B} \left( \avgint_B \left( \avgint_{B} |W(x) W^{-1}(y)|^{p'}_{\op} \,dy \right)^{\frac{p}{p'}} \,dx \right)^{\frac{1}{p}}  < \infty.
\]
Define  the classes $\A_{1,\B}$ and $\A_{\infty,\B}$ by
\[
[W]_{A_{1, \B}} = \sup_{B\in\B} \esssup_{x \in B} \avgint_B  |W^{-1}(x) W(y)|_{\op} \,dy < \infty,
\]
and
\[
[W]_{A_{\infty, \B}} = \sup_{B\in\B} \esssup_{x \in B} \avgint_B |W(x) W^{-1}(y)|_{\op} \,dy  < \infty.
\]

In the scalar case it is immediate that $w\in \A_{p,\B}$ if and only if $w^{-1}\in \A_{p',\B}$ and $[w]_{\A_{p,\B}}=[w^{-1}]_{\A_{p',\B}}$.  The same relationship, up to a constant is true, but the proof is more complicated.  To show this we introduce the concept of a reducing operator.  Given a set $E\subset \R^n$, $0<|E|<\infty$, a matrix weight $W$, $v\in \R^d$, and $1\leq p\leq <\infty$,
the quantity 
\[ \|Wv\|_{p,E} = \bigg(\avgint_E |W(x)v|^p\,dx\bigg)^{\frac{1}{p}} \]
defines a norm on $\R^d$.  Similarly, if $p=\infty$, $\|Wv\|_{\infty,E}= \||Wv|\chi_E\|_{L^\infty(\R^d)}$ defines a norm.  Then there exists a self-adjoint, positive-definite matrix $\W_E^p$ such that for all $v\in \R^d$, $\|Wv\|_{p,E}\approx |\W_E^pv|$, where the implicit constants depend only on $d$.  (This was proved for cubes in~\cite[Section~6]{bownik-cruz-uribe}; the same proof works for arbitrary sets.)  We denote the reducing operators associated to $W^{-1}$ by $\overline{\W}_E^p$.   We can characterize $\A_{p,\B}$ in terms of reducing operators.  The following result was proved in~\cite[Proposition~6.4]{bownik-cruz-uribe} for the basis of cubes, but the same proof works for any basis.

\begin{prop} \label{prop:reducing-op}
   Given a basis $\B$ and $1\leq p \leq \infty$, a matrix weight $W$ satisfies $W\in \A_{p,\B}$ if and only if 
   \begin{equation} \label{eqn:reducing-op1}
[W]_{\A_{p,\B}}^R = \sup_{B\in \B} |\W_B^p \overline{\W}_B^{p'}|_{\op} < \infty.  
\end{equation}
   Moreover, $[W]_{\A_{p,\B}}^R \approx [W]_{\A_{p,\B}}$, with implicit constants dependent only on $d$.  
   \end{prop}

As an immediate consequence of Proposition~\ref{prop:reducing-op} and Lemma~\ref{SelfAdjointCommutes} we have the following.

\begin{corollary} \label{cor:duality}
    Given a basis $\B$ and $1\leq p \leq \infty$, a matrix weight $W$ satisfies $W\in \A_{p,\B}$ if and only if $W^{-1}\in \A_{p',\B}$ and in this case $[W]_{\A_{p,\B}}\approx [W^{-1}]_{\A_{p',\B}}$.
\end{corollary}

   \medskip

   We now define the $\A_{p,q,\B}$ classes.
Given $1<p\le q<\infty$,  $W\in \A_{p,q,\B}$ if
\[
[W]_{A_{p,q, \B}} = \sup_{B\in\B} \left( \avgint_B \left( \avgint_{B} |W(x) W^{-1}(y)|^{p'}_{\op} \,dy \right)^{\frac{q}{p'}} \,dx \right)^{\frac{1}{q}}  < \infty.
\]
When $p=1$ and $q<\infty$, 
  $W \in \A_{1,q,\B}$ if
\[
[W]_{\A_{1, q, \B}} = \sup_{B\in \B} \esssup_{x \in B} \left( \avgint_B  |W^{-1}(x) W(y)|_{\op}^q \,dy \right)^{\frac1q} < \infty.
\]
Finally, when $p>1$ and $q=\infty$, $W\in \A_{p,\infty,\B}$ if
\[
[W]_{\A_{p, \infty, \B}} = \sup_{B\in \B} \esssup_{x \in B} \left( \avgint_B  |W(x) W^{-1}(y)|_{\op}^{p'} \,dy \right)^{\frac{1}{p'}} < \infty.
\]

We now consider the relationship between the  $\A_{r,\B}$  and the $\A_{p,q,\B}$ classes.  In the scalar case (i.e., when $d=1$), we have that $w\in \A_{p,q,\B}$ and $w^s\in \A_{r,\B}$ are equivalent conditions.  To see this, fix $1\leq p\leq q\leq \infty$,  define $s$ by $\frac{1}{p}-\frac{1}{q}=\frac{1}{s'}$, and define $r=\frac{q}{s}$.  We first assume that $1<p<q<\infty$.  Since $\frac{1}{p} - \frac{1}{q} = 1- \frac{1}{s}$, we have that $1<q<s$; moreover, we have that $\frac{1}{s} - \frac{1}{q} = 1- \frac{1}{p}=\frac{1}{p'}$,
which in turn implies that $ \frac{q}{q-s}=\frac{p'}{s}$. Thus, since $r'=\frac{q}{q-s}$, we  have that
$r' = \frac{p'}{s}$. Given this, the equivalence is immediate:
\begin{multline*} 
[w]_{\A_{p,q,\B}} 
= 
\sup_{B\in \B} \Bigg(\int_B w(x)^q\,dx\bigg)^{\frac{1}{q}}
\Bigg(\int_B w(x)^{-p'}\,dx\bigg)^{\frac{1}{p'}} \\
=
\sup_{B\in \B} \Bigg(\int_B w^s(x)^{\frac{q}{s}}\,dx\bigg)^{\frac{1}{q}}
\Bigg(\int_B w^s(x)^{-\frac{p'}{s}}\,dx\bigg)^{\frac{1}{p'}}
=
[w^s]_{\A_{r,\B}}^{\frac{1}{s}}.
\end{multline*}
A similar argument holds if $p=1$ or if $q=\infty$.  

In the matrix case, however, we can only prove one direction and a weaker converse.

\begin{prop}\label{prop:Mcond-Apq&Ar}
Fix $1\leq p\leq q\leq \infty$, and define $s$ by $\frac{1}{p}-\frac{1}{q}=\frac{1}{s'}$.  Let $r=\frac{q}{s}$.  If $W^s \in \A_{r,\B}$, then $W\in \A_{p,q,\B}$, and $[W]_{\A_{p,q,\B}}^s\leq [W^s]_{\A_{r,\B}}$.    Conversely, if $W\in \A_{p,q,\B}$, then $W\in \A_{r,\B}$ and $[W]_{\A_{r,\B}}\leq [W]_{\A_{p,q,\B}}$. 
\end{prop}

\begin{proof} 
If $p=q$, then $s'=\infty$; thus, $s=1$ and in this case the result is trivially true.  Therefore, we assume that $p<q$.  Suppose first that  $1<p<q<\infty$.  If $W^s \in \A_{r,\B}$, then  by  Cordes inequality (Lemma~\ref{lemma:cordes}) and the relationship between $p$, $q$, and $r$,
\begin{multline*}
    [W^s]_{\A_{r, \B}} 
    = \sup_{B\in \B} \left( \avgint_B \left( \avgint_B |W^s(x) W^{-s}(y)|_{\op}^{r'} \,dy \right)^{\frac{r}{r'}} \,dx \right)^{\frac{1}{r}}\\
     \geq \sup_{B\in \B} \left( \avgint_B \left( \avgint_B |W(x) W^{-1}(y)|_{\text{op}}^{p'} \,dy \right)^{\frac{q}{p'}} \,dx \right)^{\frac{s}{q}}
     = [W]_{\A_{p, q, \B}}^s.
\end{multline*}
On the other hand, if $W \in \A_{p, q, \B}$,  then, since $1 < s < q$, by H\"older's  inequality with exponent $s$, we have that
\begin{align*}
    [W]_{\A_{r, \B}} 
     & = \sup_{B\in \B} \bigg( \avgint_B \left( \avgint_B  |W(x) W^{-1}(y)|_{\op}^{r'} \,dy \right)^{\frac{r}{r'}} \,dx \bigg)^{\frac{1}{r}} \\
    & = \sup_{B\in \B} \bigg( \avgint_B \left( \avgint_B  |W(x) W^{-1}(y)|_{\op}^{\frac{p'}{s}} \,dy \right)^{\frac{q}{p'}} \,dx \bigg)^{\frac{s}{q}} \\
    & \le \sup_{B\in \B} \bigg( \avgint_B \left( \avgint_B  |W(x) W^{-1}(y)|_{\op}^{p'} \,dy \right)^{\frac{q}{p'}} \,dx \bigg)^{\frac{1}{q}} \\
    & = [W]_{\A_{p, q, \B}}.
\end{align*}

Now suppose $p=1$ and $1<q<\infty$; in this case $q=s$.   If  $W^q \in \A_{1,\B}$, then again by Cordes inequality,
\begin{multline*}
    [W^q]_{\A_{1, \B}} 
    = \sup_{B\in \B} \esssup_{x \in B} \avgint_B |W^q(x) W^{-q}(y)|_{\op}\,dy\\
     \geq \sup_{B\in \B} \esssup_{x \in B} \avgint_B |W(x) W^{-1}(y)|_{\op}^q \,dy = [W]_{\A_{1,q,\B}}^q.
\end{multline*}
Conversely, if $W\in \A_{1,q}$, then by H\"older's inequality,
\begin{multline*}
[W]_{\A_{1,q,\B}} = \sup_{B\in \B} \esssup_{x \in B} \bigg(\avgint_B |W(x) W^{-1}(y)|_{\op}^q \,dy\bigg)^{\frac1q} \\
\geq \sup_{B\in \B} \esssup_{x \in B} \avgint_B |W(x) W^{-1}(y)|_{\op} \,dy = [W]_{\A_{1,\B}}. 
\end{multline*}

Finally, suppose $1 \leq p < \infty $ and $q=\infty$; in this case $s=p'$.   We can essentially repeat the above arguments when $p=1$ to show that  if $ W^{p'} \in \A_{\infty, \B} $, then $W \in \A_{p,\infty,\B}$, and if 
 $W\in \A_{p,\infty,\B}$, then $W\in \A_{\infty,\B}$.
\end{proof}

The converse implication that $W\in \A_{p,q,\B}$ implies that $W^s\in \A_{r,\B}$ is false for the basis of dyadic cubes.  The reason is, essentially, that the converse of the Cordes inequality is false in general.  

\begin{example} \label{example:no-reverse-cordes}
    Fix $1\leq p\leq q< \infty$, and define $s$ by $\frac{1}{p}-\frac{1}{q}=\frac{1}{s'}$.  Let $r=\frac{q}{s}$.  
    Then there exists a $2\times 2$ matrix such that $W\in \A_{p,q,\D}$ but $W^s\not\in \A_{r,\D}$.
\end{example}

\begin{proof}
    We will show that we can modify an example due to Kakaroumpas, {\em et al.}~\cite[Prop.~6.8]{KNV} to get our desired example.  Following this paper, we say that a matrix weight $V$ is in $A_{p,q,\D}$ if
    \[ [V]_{A_{p,q,\D}}
    = 
    \sup_{Q\in \D} \avgint_Q \bigg( \avgint_Q |V^{\frac{1}{q}}(x)V^{-\frac{1}{q}}(y)|_{\op}^{p'}\,dy\bigg)^{\frac{q}{p'}} < \infty. 
\]
It is immediate that the matrix weight $W=V^{\frac{1}{q}}$ is in $\A_{p,q,\D}$ if and only if $V\in A_{p,q,\D}$, and $[V]_{A_{p,q,\D}} = [W]_{\A_{p,q,\D}}^q$. In~\cite[Prop.~6.8]{KNV} they proved the following result.  Given $1<p<q<\infty$, there exists a $2\times 2$ matrix weight $V\in A_{p,q,\D}$ such that, given any $1<p_2\leq q_2<\infty$ with $p<p_2$ and satisfying 
\begin{equation} \label{eqn:knv1}
\frac{q}{p'}=\frac{q_2}{p_2'},
\end{equation}
$V\not \in A_{p_2,q_2,\D}$.  

To apply this result, let $p_2=q_2=r$.  Then we need to show that $p<p_2$ and~\eqref{eqn:knv1} holds.  Since $\frac{1}{p}-\frac{1}{q}=1-\frac{1}{s}$, a straightforward calculation shows that 
\begin{equation} \label{eqn:knv2}
    r=\frac{q}{s}=\frac{q}{p'}+1.
\end{equation}  
Therefore,
\[ \frac{q}{p'} = r-1 = \frac{q_2}{q_2'}= \frac{q_2}{p_2'}.  \]
This gives us~\eqref{eqn:knv1}.  It is immediate from the definition that $q_2<q$. But then from~\eqref{eqn:knv2} we get (using that $q_2=p_2=r$) that
\[ 1 = \frac{1}{p_2} + \frac{q}{q_2}\Big(1-\frac{1}{p}\Big) > \frac{1}{p_2} + 1-\frac{1}{p}. \]
If we rearrange terms we get $\frac{1}{p}>\frac{1}{p_2}$, and so $p<p_2$.  

Since $V\in A_{p,q,\D}$, if we let $W=V^{\frac{1}{q}}$, then $W\in \A_{p,q,\D}$.  On the other hand, since $V\not\in A_{p_2,q_2,\D}$, we can rewrite this condition as
\[ \sup_{Q\in \D} \avgint_Q \bigg( \avgint_Q 
|W^{\frac{q}{q_2}}(x)W^{-\frac{q}{q_2}}(y)|_{\op}^{p_2'}\,dy\bigg)^{\frac{q_2}{p_2'}}\,dx =\infty. 
\]
But this is equivalent to 
\[ [W^s]_{\A_{r,\D}}^r = \sup_{Q\in \D} \avgint_Q \bigg( \avgint_Q |W^s(x)W^{-s}(y)|_{\op}^{r'}\,dy\bigg)^{\frac{r}{r'}}\,dx =\infty. 
\]
This completes the proof.
\end{proof}

We conclude this section with a result that characterizes the relationship between 
 $\A_{p,q,\B}$ weights and $\A_{p,\B}$ weights for bases that satisfy a reverse H\"older inequality. We include this result partly because it motivates an open question related to Example~\ref{example:no-reverse-cordes}.  
 We first define the reverse H\"older inequality as adapted to scalar $\A_{p,\B}$ weights.

\begin{definition}
    Given a basis $\B$ and $1\leq p<\infty$, a scalar weight $w$ satisfies the reverse H\"older inequality with respect to $\B$ with exponent $s>1$ if
    \[ [w]_{RH_{p,s,\B}} = \sup_{B\in \B} 
    \bigg(\avgint_B w(y)^{sp}\,dy\bigg)^{\frac{1}{sp}}
    \bigg(\avgint_B w(y)^p\,dy\bigg)^{-\frac{1}{p}} < \infty.
 \]
 We denote this by $w\in RH_{p,s,\B}$.  We say that $w\in RH_{p,\infty,\B}$ if
 \[ [w]_{RH_{p,\infty,\B}} = \sup_{B\in \B} 
    \esssup_{x\in B} w(x)
    \bigg(\avgint_B w(y)^p\,dy\bigg)^{-\frac{1}{p}} < \infty.\]
\end{definition}

For a number of bases, if $w\in \A_{p,\B}$, $1\leq p \leq \infty$, then $w\in RH_{p,s,\B}$ for some $s>1$; the value of $s$ depends on $[w]_{\A_{p,\B}}$.  These include the basis $\Q$ of cubes and the basis $\D$ of dyadic cubes~\cite[Section~7.2]{duoandikoetxea01}, the basis $\Rdf$ of rectangles~\cite[Section~7]{MR3473651}, and the Zygmund basis~$\Zz$~\cite[p.~82]{MR864371}.  However, not every basis satisfies has this property:  see~\cite[Section~5]{MR3473651}.  The class $RH_\infty$ was introduced for the basis of cubes in~\cite{cruz-uribe-neugebauer95}; it is in some sense "dual" to the class $A_1$ and plays a role in the Jones factorization theorem for $A_p$ weights:  see~\cite{preprint-DCU}.

We now extend this definition to matrix weights.  To motivate our definition,  recall that if $W\in \A_{p,\Q}$, then $|W|_{\op}$ is in the scalar $\A_{p,\Q}$ class: see~\cite[Corollary~6.7]{bownik-cruz-uribe}.  Therefore, for some $s>1$, $|W|_{\op} \in RH_{p,s,\Q}$.  

\begin{definition}\label{defn_6.4}
    Given a matrix weight \( W \),  \( 1 \leq p < \infty \), and \( 1 < s \leq \infty \), we say that $W\in RH_{p,s,\B}$ if 
    $$
    [W]_{RH_{p,s,\B}} = \sup_{v \in \R^d\setminus\{0\}} \left[ |W v| \right]_{RH_{p,s,\B}} < \infty.
    $$
\end{definition}

With the reverse H\"older condition we can completely characterize the relationship between the $\A_{p,q,\B}$ condition and the $\A_{p,\B}$ condition.  The following result was  proved in~\cite[Lemma~6.5]{KNV} for the basis of cubes.  The proof for a general basis is essentially the same and so we omit the details. 

\begin{prop}\label{lemma_6.5}
    Let $W\in \Ss_d$, \( 1 \leq  p < q < \infty \) and $s=\frac{q}{p}$. Then the following statements hold:
    \begin{enumerate}[(i)]
        \item \( W \in \A_{p,\B} \cap RH_{p,s,\B} \) if and only if \( W \in \A_{p,q,\B} \); moreover, 
        \[ \max \Big\{ [W]_{A_{p,\B}}, [W]_{RH_{p,s,\B}} \Big\} 
        \lesssim [W]_{\A_{p,q,\B}} 
        \lesssim  [W]_{\A_{p,\B}} [W]_{RH_{p,s,\B}}.\]
        
        \item For $q=\infty$, $W \in \A_{p,\B} \cap RH_{p,\infty,\B}$ if and only if $W\in \A_{p,\infty,\B}$; moreover,\[ \max \left\{ [W]_{\A_{p,\B}}, [W]_{RH_{p,\infty,\B}} \right\} 
        \lesssim [W]_{\A_{p,\infty,\B}}^p \lesssim [W]_{RH_{p,\infty\,B}} [W]_{\A_{p,\B}}^p. \]
    \end{enumerate}
   The implicit constants only depend on $d$, $p$, and $q$.
\end{prop}

Given Proposition~\ref{lemma_6.5}, it is an open question whether there exists a condition analogous to the reverse H\"older inequality which, combined with the assumption that $W\in \A_{p,q,\B}$, would imply that $W^s\in \A_{r,\B}$.  Obviously, if we assume that $W$ satisfies a reverse Cordes inequality of the form
\[ |W^s(x)W^{-s}(y)|_{\op}\leq C|W(x)W^{-1}(y)|_{\op}^s, \]
where $C$ is independent of $x,\,y\in \R^n$,
then we could adapt the proof of Proposition~\ref{prop:Mcond-Apq&Ar} to prove this.  Such an assumption seems artificial, but we have not been able to find any more satisfactory condition that implies this.  One possibility is that some condition related to the matrix $A_\infty$ condition introduced by Volberg~\cite{Vol} would be sufficient.

\section{Convex-set valued operators defined on bases}
\label{section:tech-lemmas}

In this section we gather a number of definitions and results about convex-set valued functions and operators needed for the proof of our main result.  These are all drawn from~\cite{bownik-cruz-uribe} and we refer the reader there for proofs and more information.  In this paper everything was done working with cubes in $\R^n$.  In many cases, the proofs go over without change to an arbitrary basis, in which case we will just state the result.  However, if a new or modified proof is required, we will give it here.  

Hereafter, let $\cs$ be the collection of convex sets $K \subset \R^d$ that are closed and symmetric, and let $\bcs$ be the  convex sets  that are bounded, closed, and symmetric.  By a convex-set valued function we mean a map $F : \R^n \rightarrow \cs$.  The function $F$ is measurable if given any open set $U\subset \R^d$, $F^{-1}(U)=\{ x \in \R^n : F(x) \cap U \neq \emptyset\}$ is a Lebesgue measurable set.  For equivalent characterizations of measurability, see \cite[Theorem~3.2]{bownik-cruz-uribe}. It is  integrably bounded if there exists a non-negative function $k\in L^1(\R^n,\R)$ such that $F(x)\subset k(x)\BB$ for almost every $x\in \R^n$.  If this is true for $k\in L^1_{loc}(\R^n,\R)$ we say $F$ is locally integrably bounded.

Given a matrix weight $W$ and $1\leq p<\infty$, define the set $L^p_{\K}(\R^n,W)$ to consist of all convex-set valued functions $F : \R^n \rightarrow \cs$ such that 
\[ \|F\|_{L^p_{\K}(\R^n,W)} = \bigg(\int_{\R^n} |W(x)F(x)|^p\,dx \bigg)^{\frac{1}{p}} < \infty,  \]
where $|W(x)F(x)|= \sup\{ |W(x)v| : v \in F(x)\}$.  If $W=I_d$ (the $d\times d$ identity), then we will denote the space by $L^p_{\K}(\R^n)$.

\begin{definition}
    Given an operator $T$ that maps convex-set valued functions to convex-set valued functions, we say that $T$ is sublinear if given any $F,\,G  : \R^n \rightarrow \bcs$ and $\alpha \in R$, 
    \[ T(F+G)(x) \subset TF(x) + TG(x), \quad \text{ and } \quad T(\alpha F)(x) = \alpha TF(x).  \]
    We say that $T$ is monotone if, whenever $F(x)\subset G(x)$ for all $x\in \R^n$, $TF(x)\subset TG(x)$.  
\end{definition}

\begin{definition}\label{defn_8.3}
    Let $W$ be an invertible matrix weight. Given a measurable convex-set valued function $H: \R^n \to \K_{bcs}(\R^d)$, define the exhausting operator $N_{W}$ with respect to $W$ by
\[
N_{W} H(x) = \left| W(x) H(x) \right| W(x)\overline{\mathbf{B}}.
\]
\end{definition}

\begin{lemma}\cite[Lemma~8.4]{bownik-cruz-uribe}\label{lemma_8.4}
    Given  $1 < p < \infty $, a matrix $W$,  and $ H \in L^p_{\K}(\R^n, W) $, the the operator $N_W$ satisfies:
\begin{enumerate}
    \item $ H(x) \subset N_{W} H(x) $;
    
    \item $ N_{W} $ is an isometry: 
    $\| N_{W} H \|_{L^p_{\K}(\R^n, W)} = \| H \|_{L^p_{\K}(\R^n, W)}$;

    \item $N_{W} $ is sublinear and monotone.
\end{enumerate}
\end{lemma}

\medskip

We now define the convex-set valued maximal operator  with respect to a general basis.  Recall that the Aumann integral of a locally integrably bounded function $F : \R^n \rightarrow \bcs$ on a bounded set $\Omega$ is the convex set defined by
\[ \int_\Omega F(y)\,dy = \bigg\{ \int_\Omega f(y)\,dy : f \in S^1(\Omega) \bigg\},\]
where $S^1(\Omega)$ is the collection of all integrable selection functions of $F$: that is, vector-valued functions $f\in L^1(\Omega)$ such that $f(y)\in F(y)$ for all $y\in \Omega$. 

\begin{definition} \label{defn:convex-set-max-op}
    Given a basis $\B$, and a locally integrably bounded function $F : \R^n \rightarrow \bcs$, define the maximal operator $M_\B$ by
    \[ M_\B F(x) = \clconv\bigg(\bigcup_{B\in \B} \avgint_B F(y)\,dy \cdot \chi_B(x) \bigg). \]
    Here $\clconv(E)$ is the closed, convex hull of $E$.
\end{definition}

It is immediate that $M_\B F : \R^n \rightarrow \cs$.  It is a measurable function as well.  This requires a different proof than the one given in \cite{bownik-cruz-uribe} for the basis of cubes, since in that case it is possible to replace the collection of all cubes by cubes whose vertices have rational coordinates. For the proof we need a lemma.

\begin{lemma} \label{lemma:conv-cl}
Given a set $G \subset \R^d$, $\clconv(\overline{G})=\clconv(G)$.
\end{lemma}

\begin{proof}
    It is immediate that $\clconv(G)\subset \clconv(\overline{G})$, so we only need to prove the reverse inclusion.  Fix $z\in \clconv(\overline{G})$ and fix $\epsilon>0$.  Then there exists $w\in \conv(\overline{G})$ such that $|z-w|<\epsilon/2$.  Moreover, there exist points $w_1,\ldots,w_k \in \overline{G}$ and $a_i \geq 0$, $\sum a_i =1$, such that
    \[ w = \sum_{i=1}^k a_i w_i. \]
    But then for each $i$ there exist points $v_i \in G$ such that $|v_i-w_i|<\epsilon/2$.  Therefore, if we define
    \[ v = \sum_{i=1} a_i v_i, \]
    we have that $v\in \conv(G)$ and $|z-v| \leq |z-w|+|w-v| < \epsilon$.  Since this is true for every $\epsilon$, $z\in \clconv(G)$, as desired.
\end{proof}

\begin{theorem} \label{thm:measurable}
    Given a basis $\B$, and a locally integrably bounded function $F : \R^n \rightarrow \bcs$, $M_\B$ is a measurable function.
\end{theorem}

\begin{proof}
    To prove this, we need the following result from~\cite[Theorem~8.2.2]{MR2458436}:  if $\overline{G}$ is a closed-set valued function defined on $\R^n$,  then the map $H : \R^n \rightarrow \cs$ defined by $H(x) = \clconv \big(\overline{G}(x)\big)$,  is measurable.  Define $\overline{G}(x) = \clconv (G(x))$, where
    \[ G(x) = \bigg(\bigcup_{B\in \B} \avgint_B F(y)\,dy \cdot \chi_B(x) \bigg). \]
    Then we would be done if  $H(x)=M_\B F(x)$.  But this follows immediately from Lemma~\ref{lemma:conv-cl}.
\end{proof}

\begin{lemma} \cite[Lemma~5.5]{bownik-cruz-uribe} \label{lemma_5.5}
    Given a basis $\B$, the convex-set valued maximal operator $M_\B$ is sublinear and monotone.
\end{lemma}

The proof of the following result is the same as the proof of \cite[Theorem~6.9]{bownik-cruz-uribe}, except that instead of using the boundedness of the Christ-Goldberg maximal operator, we use the assumption that $\B$ is a Muckenhoupt basis.

\begin{theorem}\label{thm_6.9}
  Let $\B$ be a Muckenhoupt basis.  Fix $1<p\leq \infty$   Given a matrix $ W \in A_{p,\B} $, $ 1 < p < \infty $, the convex-set valued maximal operator $ M_\B: L^p_{\K}(\R^n, W) \to L^p_{\K}(\R^n, W) $ is bounded:  
  \[ \|M_\B F\|_{L^p_{\K}(\R^n, W)} \leq C(d)C_\B([W]_{\A_p})\|MF\|_{L^p_{\K}(\R^n, W)}. \]
\end{theorem}

\begin{remark}
If $ \B $ is $ \mathcal{Q} $, the basis of cubes in $\R^n$, then by~\cite[Theorem~6.9]{bownik-cruz-uribe},
\[
\| M F \|_{L^p_{\K}(\R^n, W)} \leq C(n, d, p) [W]_{\A_p}^{p'} \| F \|_{L^p_{\K}(\R^n, W)}.
\]
\end{remark}

%%%%%%%

The following theorem is proved in \cite[Theorem~7.6,Lemma~8.5]{bownik-cruz-uribe}.
\begin{theorem}\label{thm_7.6}
   Given a matrix weight and  $1 \leq p < \infty$, suppose $T$ is a convex-set valued operator with the following properties:
\begin{enumerate}
    \item $T : L^p_\K(\R^n, W) \to L^p_\K(\R^n, W)$ with norm $\|T\|_{L^p_\K(\R^n, W)}$;
    \item $T$ is sublinear and monotone in the sense of Lemma \ref{lemma_5.5}.
\end{enumerate}

Given $G \in L^p_\K(\R^n, G)$, define
\begin{equation}
    SG(x) = \sum_{k=0}^\infty 2^{-k} \|T\|_{L^p_\K(\R^n, W)}^{-k} T^k G(x),
\end{equation}
where $T^k G = T \circ T \circ \cdots \circ T G$ for $k \geq 1$ and $T^0 G(x) = G(x)$. Then this series converges in $L^p_\K(\R^n, W)$ and $SG : \R^n \to \K_{bcs}(\R^d)$ is a measurable mapping. Moreover, it has the following properties:
\begin{enumerate}
    \item $G(x) \subset SG(x)$;
    \item $\|SG\|_{L^p_\K(\R^n, \rho)} \leq 2 \|G\|_{L^p_\K(\R^n, W)}$;
    \item $T(SG)(x) \subset 2 \|T\|_{L^p_\K(\R^n, W)} SG(x)$.
\end{enumerate}
Moreover, if $TG$ is an ellipsoid valued function whenever $G$ is, then $SG$ is also ellipsoid-valued.
\end{theorem}

The following results allow us to apply "reverse" factorization in our proof to construct a matrix $\A_{p,\B}$ weight.

\begin{definition}\cite[Definition 7.1]{bownik-cruz-uribe}\label{defn_A_1^K}
    Given a locally integrable bounded function $F: \R^n \to \K_{bcs}(\R^d)$, we say that it is in  convex-set valued $\A_{1,\B}^{\K}$ if there exists a constant $C$ such that for almost every x,
    \[
    M_\B F(x) \subset CF(x).
    \]
    The infimum of all such  constants is denoted by $[F]_{\A_{1,\B}^\K}$.
\end{definition}

\begin{prop}\cite[Corollary 7.4]{bownik-cruz-uribe}\label{cor_A_1^K}
  Given a matrix weight $W$, $W\in \A_{1,\B}$ if and only if $W\overline{\mathbf{B}}\in \A_{1,\B}^\K$.  Moreover, $[W]_{\A_{1,\B}} \approx [W\overline{\mathbf{B}}]_{A_{1,\B}}^\K$, where the implicit constants depend only on $d$.  
\end{prop}

\begin{prop}\cite[Proposition 8.7]{bownik-cruz-uribe}\label{prop_8.7}
  Given \( W_0 \in \A_{1,\B} \), \( W_1 \in \A_{\infty,\B} \), suppose $W_0,\,W_1$ commute. Then for \( 1 < p < \infty \), we have \( W_0 ^{\frac{1}{p}} W_1^{\frac{1}{p'}} \in \A_{p,\B} \).  Moreover,
  \[ [W]_{\A_{p,\B}} \leq c(d) [W_0]_{\A_{1,\B}}^{\frac{1}{p}} [W_1]_{\A_{\infty,\B}}^{\frac{1}{p'}}.  \]
\end{prop}

\begin{remark}
    In \cite{bownik-cruz-uribe} a more general "reverse" factorization result is proved which does not assume the matrices commute.  It remains true in this more general setting, but we only give the version needed to prove Theorem~\ref{thm:off-diag-extrapol}.
\end{remark}

\section{Proof of Theorem~\ref{thm:off-diag-extrapol}}
\label{section:main-thm}

In this section we prove our main extrapolation result, Theorem~\ref{thm:off-diag-extrapol}.  In fact, we will prove a more general result, stated in terms of a family $\F$ of extrapolation pairs.    Hereafter, $\F$ will denote a family of pairs $(f,g)$ of
measurable, vector-valued functions such that neither $f$ nor
$g$ is equal to $0$ almost everywhere.  If we write an
inequality of the form
\[ \|f\|_{L^p(\R^n,W)} \leq C\|g\|_{L^p(\R^n,W)}, \qquad (f,g) \in
    \F, \]
  we mean that this inequality holds for all pairs $(f,g)\in\F$ such
  that the lefthand side of this inequality is finite.  The constant,
  whether given explicitly or implicitly, is assumed to be independent
  of the pair $(f,g)$ and to depend only on $[W]_{\A_p}$ and not on
  the particular weight $W$.  To prove Theorem~\ref{thm:off-diag-extrapol} from Theorem~\ref{thm:off-diag-extrapol-pairs} below, we need to construct the family $\F$.  Intuitively, it suffices to take $\F = \{ (|Tf|,|f|) \}$, where the functions $f$ are taken from some suitable dense subset of $L^p(\R^n,W)$.  However, it  is important to remember that $\|Tf\|_{L^p(\R^n,W)}<\infty$ is a necessary  assumption and insuring this requires some minor additional work.     In the scalar case this can easily be done via
  a truncation argument and approximation:
  see~\cite[Section~6]{preprint-DCU}.  In the case of matrix weights a
  similar argument can be applied: see \cite[Section~10]{bownik-cruz-uribe} for a detailed discussion.

\begin{theorem} \label{thm:off-diag-extrapol-pairs}
Let $\B$ be a matrix Muckenhoupt basis and let $\F$ be a family of extrapolation pairs.   Fix $s$, $1\leq s<\infty$. Suppose that for some $ p_0, q_0, 1 \leq p_0 \leq q_0 \leq \infty $, where $\frac{1}{p_0}-\frac{1}{q_0}= \frac{1}{s'}$ and $r_0=\frac{q_0}{s}$, there exists a positive increasing function $ N_{p_0,q_0}$ such that for every matrix weight $W$ with $ W_0^s \in \A_{r_0,\B} $,
\begin{multline} \label{eqn:extrapol0}
    \left( \int_{\R^n} \left| W_0(x) f(x) \right|^{q_0} \,dx \right)^{\frac{1}{q_0}} \\
\leq N_{p_0, q_0} \left( [W_0^s]_{\A_{r_0,\B}}\right)
\left( \int_{\R^n} \left| W_0(x) g(x) \right|^{p_0} \,dx \right)^{\frac{1}{p_0}}, \quad (f,g) \in \F.
\end{multline}
Then for all $ p,\, q $ such that $ 1 < p \leq q < \infty $, $ \frac{1}{p} - \frac{1}{q} = \frac{1}{p_0} - \frac{1}{q_0}$, $r=\frac{q}{s}$, and for all matrix weights $W$ with $ W^s \in \A_{r,\B} $,
\begin{multline} \label{eqn:extrapol1}
\left( \int_{\R^n} \left| W(x) f(x) \right|^q \,dx \right)^{\frac{1}{q}} \\
\leq N_{p, q}(\B,p_0,q_0,p,q,[W^s]_{\A_{r,\B}})
\left( \int_{\R^n} \left| W(x) g(x) \right|^p \,dx \right)^{\frac{1}{p}}, \qquad (f,g) \in \F,
\end{multline}
where 
\begin{multline*} N_{p,q}(\B,p_0,q_0,p,q, [W^s]_{\A_{r,\B}}) \\
    =
C(d,p_0,q_0,p,q) N_{p_0,q_0}\big( C(d,p_0,q_0) 
K_{r'}([W^{-s}]_{\A_{r',\B}})^{\frac{1}{r_0}} K_r([W^s]_{\A_{r,\B}})^{\frac{1}{r_0'}}\big). 
\end{multline*}
\end{theorem}

\begin{remark}
    The constant in Theorem~\ref{thm:off-diag-extrapol-cubes} follows immediately from that in Theorem~\ref{thm:off-diag-extrapol-pairs} and Remark~\ref{remark:sharp-constant-cubes}.
\end{remark}

\begin{proof}
For the proof,  we will  assume that $p_0<q_0$, so that $1<s<\infty$.  In the diagonal case, when $p_0=q_0$, the same proof works with minor modifications.  We also note that this case is proved in \cite[Theorem~9.1]{bownik-cruz-uribe} for the basis $\Q$ of cubes. 
   We divide the proof into three cases: $1<p_0,\, q_0<\infty$, $p_0=1$ and $q_0<\infty$, and $p_0>1$ and $q_0=\infty$.

\textbf{Case I: $\mathbf{1<p_0,\, q_0<\infty}$}. Fix $1<p_0<q_0 <\infty$, and define $p,\,q,\,r,\,s$  as in the statement of the theorem.  Fix a matrix weight $W$ such that $W^s \in \A_{r,\B}$. We first define two iteration operators. Define $ P_{W} = N_{W^s} M_\B$; then by Lemmas~\ref{lemma_8.4} and~\ref{lemma_5.5} , $ P_{W} $ is sublinear and monotone, and 
\[
\| P_{W} \|_{L^r_{\K}(\R^n, W^s)} = \| M_\B \|_{L^r_{\K}(\R^n, W^s)}= K_r([W^s]_{\A_{r,\B}}),
\]
where $K_r$ is the function from the definition of a Muckenhoupt basis (Definition~\ref{defn:matrix-muckenhoupt}).  By Theorem~\ref{thm_7.6},
\[
\mathcal{R}_W H(x) = \sum_{k=0}^{\infty} 2^{-k} \| P_{W} \|_{L^r_{\K}(\R^n, W^s)}^{-k} P_{W}^k H(x)
\]
is defined on $ L^r_{\K}(\R^n, W^s) $ and satisfies:
\begin{enumerate} 
    \item[(A$_0$)] $ H(x) \subset \mathcal{R}_W H(x) $;
    \item[(B$_0$)] $ \| \mathcal{R}_W H \|_{L^r_{\K}(\R^n, W^s)} \leq 2 \| H \|_{L^r_{\K}(\R^n, W^s)} $;
    \item[(C$_0$)] $ \mathcal{R}_W H \in \A_1^{\K} $ and 
    $M(\mathcal{R}_W H)(x) \subset  P_{W}(\mathcal{R}_W H)(x) \subset 2K_r([W^s]_{\A_{r,\B}})\mathcal{R}_W H(x)$. 
\end{enumerate}
To define the second iteration operator, note that since $ W^s \in \A_{r, \B} $, by Corollary~\ref{cor:duality},  $W^{-s}\in \A_{r',\B}$ and so $ M_\B $ is also bounded on $ L^{r'}_{\K}(\R^n, W^{-s}) $.  Let
\[
M'_s H(x) = W^{-s}(x) M_\B(W^s H)(x).
\]
Then
\[
\| M'_s H \|_{L^{r'}_{\K}(\R^n)} 
= \| M_\B(W^s H) \|_{L^{r'}_{\K}(\R^n, W^{-s})} 
\leq K_{r'}([W^{-s}]_{\A_{r',\B}}) \| H \|_{L^{r'}_{\K}(\R^n)}.
\]

Let $ P'_{I} = N_{I} M'_s $;  then again by Lemmas~\ref{lemma_5.5} and~\ref{lemma_8.4} we have that $ P'_{I} $ is sublinear, monotone and its operator norm is the same as $M'_s$.  Hence, by Theorem \ref{thm_7.6}, if we define
\[
\mathcal{R}'_{I} H(x) = \sum_{k=0}^{\infty} 2^{-k} \| P'_{I} \|_{L^r_{\K}(\R^n)}^{-k} (P'_{I})^k H(x),
\]
then we have:
\begin{enumerate}
    \item[(A$'_0$)] $ H(x) \subset \mathcal{R}'_{I} H(x) $;
    \item[(B$'_0$)] $ \| \mathcal{R}'_{I} H \|_{L^{r'}_{\K}(\R^n)} \leq 2 \| H \|_{L^{r'}_{\K}(\R^n)} $;
    \item[(C$'_0$)] $ W^s \mathcal{R}'_{I} H \in \A_1^{\K} $ and $ M(W^s \mathcal{R}'_{I} H)(x) \subset K_{r'}([W^{-s}]_{\A_{r',\B}}) W^s \mathcal{R}'_{I} H(x) $.
\end{enumerate}

\medskip

Now fix $ (f, g) \in \mathcal{F}$; by assumption we have that $0 < \|f\|_{L^q(\R^n,W)} < \infty$ and $\|g\|_{L^p(\R^n,W)}>0$; we may also assume that $\|g\|_{L^p(\R^n,W)}<\infty$, since otherwise the desired conclusion is trivial.  Define the convex-set valued functions
\[ F(x)=\clconv\{-f(x),f(x)\}, \qquad G(x)=\clconv\{-g(x),g(x)\}.  \]
Then we have that 
\[
 N_WF(x)=|W(x)F(x)|W^{-1}(x)\overline{\mathbf{B}} = |W(x)f(x)|W^{-1}(x)\overline{\mathbf{B}}; 
\]
similarly, $N_WG(x)= |W(x)g(x)|W^{-1}(x)\overline{\mathbf{B}}$.
Define the ellipsoid valued function
\[
H_1(x) = \bigg( \frac{|W(x) f(x)|}{\| f \|_{L^q(\R^n, W)}} + \frac{|W(x) g(x)|^{\frac{p}{q}}}{\| g \|_{L^p(\R^n, W)}^\frac{p}{q}} \bigg) W^{-1}(x) \overline{\mathbf{B}}= h_1(x)W^{-1}(x)\overline{\mathbf{B}};
\]
then $\| H_1 \|_{L^q(\R^n, W)} = \|h_1\|_{L^q(\R^n)} \leq 2$. 
Let $H_1^s = h_1^s W^{-s} \overline{\mathbf{B}}$; then
\begin{equation}\label{eq_H1s}
    \| H_1^s \|_{L^r_{\K}(\R^n, W^s)} 
    = \bigg(\int_{\R^n} \left| W^s(x) h^s_1(x)W^{-s}(x)\overline{\mathbf{B}} \right|^r \,dx \bigg)^{\frac{1}{r}}
    = \bigg( \int_{\R^n} \left| h_1 \right|^q \,dx \bigg)^{\frac{1}{q}s}  \leq 2^s. 
\end{equation}
Thus,  $H_1^s \in L^r_{\K}(\R^n, W^s)$ and $ \mathcal{R}_W H_1^s $ is defined. Since $H_1^s$ is  an ellipsoid-valued function, by Theorem~\ref{thm_7.6}, $\mathcal{R}_W H^s_1$ is also an ellipsoid-valued function. Therefore, there exists a scalar function $ R_W h_1^s $ such that
\[
\mathcal{R}_W H_1^s(x) = \mathcal{R}_W h_1^s W^{-s} \overline{\mathbf{B}};
\]
further, it follows  from (A$_0$) that  $ h_1^s(x) \leq R_W h_1^s(x) $. Thus, $h_1(x) \leq \left( R_W h_1^s \right)^{\frac{1}{s}}(x)= \overline{h}_1(x)$. 
We claim that properties (A$_0$), (B$_0$), and (C$_0$) imply the following:
\begin{enumerate}
    \item [(A$_1$)]  $ h_1(x) \leq \overline{h}_1(x) $;
    \item [(B$_1$)] $ \| \overline{h}_1 \|_{L^q(\R^n)} \leq 2^{\frac{1}{s}+1} $;
    \item [(C$_1$)] $ \overline{h}_1^s W^{-s} \in \A_{1,\B} $ and $[\overline{h}_1^s W^{-s}]_{\A_{1,\B}} \leq c(d)K_r([W^s]_{\A_{r,\B}}) $.
\end{enumerate}
We already proved  (A$_1$).   Inequality (B$_1$) follows from  (B$_0$) and inequality~\eqref{eq_H1s}:
\begin{multline*}
    \| \overline{h}_1 \|_{L^q(\R^n)} = \left( \int_{\R^n} \left| (\mathcal{R}_W h_1^s)^{\frac{1}{s}}(x) \right|^q \,dx\right)^{\frac{1}{q}} \\
    = \left( \int_{\R^n} \left| W^s(x) \mathcal{R}_W H_1^s(x) \right|^r \,dx \right)^{\frac{1}{rs}} 
    \leq 2^{\frac{1}{s}} \|H_1^s \|_{L^r_{\K}(\R^n,W^s)}^{\frac{1}{s}} \leq 2^{\frac{1}{s}+1}.
\end{multline*}
To prove (C$_1$): from (C$_0$) we have that
    \[
    \mathcal{R}_W H_1^s(x) = \mathcal{R}_W h_1^s W^{-s} \overline{\mathbf{B}} \in \A_{1,\B}^{\K},
    \]
    and so by Lemma~\ref{cor_A_1^K},  $H_1^s W^{-s} \in \A_{1,\B}$ and $[\overline{h}_1^s W^{-s}]_{\A_{1,\B}} \leq c(d)K(\B,W^s,r) $. 

\medskip

To estimate  the lefthand side of \eqref{eqn:extrapol1} we use duality.  Since $\| f \|_{L^q(\R^n, W)}= \| |Wf|^s \|^{\frac{1}{s}}_{L^r(\R^n)} $, by the duality of $L^r(\R^n)$, there exists $h_2(x)\in L^{r'}(\R^n)$ with $\|h_2\|_{L^{r'}(\R^n)}=1$ such that
\[
\|f\|_{L^q(\R^n,W)}=\left(\int_{\R^n} |W(x)f(x)|^s h_2(x)dx \right)^{\frac{1}{s}}.
\]
Let $H_2(x)=h_2(x)\overline{\mathbf{B}}$; then $H_2\in L^{r'}_{\K}(\R^n)$, $\|H_2\|_{L^{r'}_{\K}(\R^n)}=1$, and  $\mathcal{R}'_{I}H_2(x)$ is in $L^{r'}_{\K}(\R^n)$. Moreover, since $H_2$ is a ball-valued function, by Theorem~\ref{thm_7.6}, $\mathcal{R}'_{I}H_2$ is also ball-valued. Therefore, there exists a scalar function $\overline{h}_2= \mathcal{R}'_{I}h_2$ such that
\[
\mathcal{R}'_{I}H_2(x) =\mathcal{R}'_{I}h_2(x)\overline{\mathbf{B}}.
\]
We now claim that the following hold:
\begin{enumerate}
    \item [(A$_1'$)]  $ h_2(x) \leq \overline{h}_2(x) $;
    \item [(B$_1'$)] $ \| \overline{h}_2 \|_{L^{r'}(\R^n)} \leq 2 $;
    \item [(C$_1'$)] $ \overline{h}_2 W^{s} \in \A_{1,\B} $ and $[\overline{h}_2 W^{s}]_{\A_{1,\B}} \leq c(d)K_{r'}([W^{-s}]_{\A_{r',\B}}) $.
\end{enumerate}
Inequality (A$_1'$) follows from (A$_0'$) and the definition of $\overline{h}_2$.  Inequality (B$_1'$) follows from (B$_0'$):
\[ \| \overline{h}_2 \|_{L^{r'}(\R^n)} = 
\| \mathcal{R}'_{I} H_2 \|_{L^{r'}_{\K}(\R^n)} \leq 2 \| H_2 \|_{L^{r'}_{\K}(\R^n)} = 2. \]
Finally, (C$_1'$) follows from (C$_0'$) and Lemma~\ref{cor_A_1^K}.

\medskip

We can now prove inequality~ \eqref{eqn:extrapol1}.  
By (A$_1'$) and  H\"older's inequality with $r_0=\frac{q_0}{s}$ and $r'_0$, we have that

\begin{align*}
    \|f\|_{L^q(\R^n,W)} &= \left( \int_{\R^n} |W(x)f(x)|^s h_2(x) \,dx \right)^{\frac{1}{s}} \\
    & \le \left( \int_{\R^n} |W(x)f(x)|^s \overline{h}_2(x) \,dx \right)^{\frac{1}{s}}\\
    & = \left( \int_{\R^n} |W(x)f(x)|^s\overline{h}_1^{-\frac{s}{r'_0}}(x) \overline{h}_1^{\frac{s}{r'_0}}(x) 
    \overline{h}_2(x) \,dx \right)^{\frac{1}{s}}\\
    & \le \left( \int_{\R^n} |W(x)f(x)|^{q_0} \overline{h}_1^{-\frac{sr_0}{r'_0}}(x)  \overline{h}_2(x) \,dx \right)^{\frac{1}{r_0s}} 
    \left( \int_{\R^n} \overline{h}_1^s(x) \overline{h}_2(x) \,dx \right)^{\frac{1}{r'_0s}} \\
    & = I_1\cdot I_2^{\frac{1}{r'_0s}}.
\end{align*}

We first estimate $I_2$.  By H\"older's inequality for $ r$ and $ r' $, (B$_1$) and (B$_1'$),
\begin{equation*}
    I_2  
     =  \int_{\R^n} \overline{h}_1^s(x) \overline{h}_2(x) \,dx  
     \leq \left( \int_{\R^n} \overline{h}_1(x)^q \,dx \right)^{\frac{1}{q}s} 
    \left( \int_{\R^n}  \overline{h}_2(x)^{r'} \,dx \right)^{\frac{1}{r'}}
       \leq  2^{s+2}.
\end{equation*}

We now estimate $I_1$.  By (A$_1$) and the definition of $h_1$, we have that
\[
\frac{|W(x) f(x)|}{\| f \|_{L^q(\R^n, W)}} \leq h_1(x) \leq \overline{h}_1(x).
\]
Therefore,
\begin{multline*}
    I_1^{q_0} = \left( \int_{\R^n} |W(x) f(x)|^{q_0} \overline{h}_1^{-(q_0 - s)}(x)\overline{h}_2(x) \,dx \right) \\
    \leq \| f \|_{L^{q}(\R^n, W)}^{q_0} \left( \int_{\R^n} \overline{h}_1^{s}(x) \overline{h}_2(x) \,dx \right)
= \| f \|_{L^{q}(\R^n, W)}^{q_0} I_2 < \infty.
\end{multline*}
Define the matrix $W_0$ by 
\begin{equation*}
W_0 =  \overline{h}_1(x)^{-\frac{q_0 - s}{q_0}} \overline{h}_2(x)^{\frac{1}{q_0}} W(x).
\end{equation*}
Assume for the moment that $W_0^{s} \in \A_{r_0,\B}$.  
Again by  (A$_1$) and the definition of $h_1$, we have that 
\[
\frac{|W(x) g(x)|^{\frac{p}{q}}}{\| g \|_{L^p(\R^n, W)}^{\frac{p}{q}}} \leq h_1(x) \leq \overline{h}_1(x).
\]
Therefore, we can apply our hypothesis \eqref{eqn:extrapol0} to get
\begin{align*}
    I_1 & = \left( \int_{\R^n} |W_0(x) f(x)|^{q_0} \,dx \right)^{\frac{1}{q_0}} \\
    & \leq N_{p_0, q_0}([W_0^s]_{\A_{r_0,\B}}) \left( \int_{\R^n} |W_0(x) g(x)|^{p_0} \,dx \right)^{\frac{1}{p_0}}\\
    & = N_{p_0, q_0}([W_0^s]_{\A_{r_0,\B}}) \left( \int_{\R^n} \left| W(x)g(x)\right|^{p_0} \overline{h}_1(x)^{-\frac{(q_0 - s)p_0}{q_0}} \overline{h}_2(x)^{\frac{p_0}{q_0}}  \,dx \right)^{\frac{1}{p_0}}\\
    & \leq N_{p_0, q_0}([W_0^s]_{\A_{r_0,\B}})  \| g \|_{L^{p}(\R^n, W)} 
    \left( \int_{\R^n} \overline{h}_1(x)^{ \frac{p_0q}{p} - \frac{(q_0-s)p_0}{q_0} } \overline{h}_2(x)^{\frac{p_0}{q_0}} \,dx \right)^{\frac{1}{p_0}}.
\end{align*}

Let $ \alpha = r'\frac{q_0}{p_0} > 1 $. Then
\[
\alpha'\cdot\left( \frac{p_0 q}{p} - \frac{p_0 (q_0 - s)}{q_0} \right) = q,
\]
and so, by H\"older's inequality for $\alpha$ and $\alpha'$, (B$_1$) and (B$_1'$), 
\begin{align*}
    I_1 & \leq N_{p_0, q_0}([W_0^s]_{\A_{r_0,\B}}) \| g \|_{L^{p}(\R^n, W)} 
    \left( \int_{\R^n} \overline{h}_1(x)^{q} \,dx \right)^{\frac{1}{p_0\alpha'}} 
    \left( \int_{\R^n} \overline{h}_2(x)^{r'} \,dx \right)^{\frac{1}{r'q_0}}\\
       & \leq 2^{\left[\left(1 + \frac{1}{s}\right) \frac{q}{p_0\alpha'} + \frac{1}{q_0}\right]} 
    N_{p_0, q_0}([W_0^s]_{\A_{r_0,\B}}) \| g \|_{L^{p}(\R^n, W)}.
\end{align*}
Hence,
\begin{equation*}
    \|f\|_{L^q(\R^n,W)}  \le I_1 I_2^{\frac{1}{r'_0 s}} \le 2^{\left[\left(1 + \frac{1}{s}\right) \frac{q}{p_0\alpha'} + \frac{1}{q_0}+ \frac{s+2}{sr'_0}\right]} N_{p_0, q_0}([W_0^s]_{\A_{r_0,\B}}) \| g \|_{L^{p}(\R^n, W)}.
\end{equation*}

\medskip

Therefore, to complete the proof we need to prove that $W_0^s \in \A_{r_0,\B}$ and estimate $[W_0^s]_{\A_{r_0,\B}}$.  Define
\[  V_1 = \overline{h}_2 W^s \quad \text{ and } \quad  V_2 = \overline{h}_1^{-s} W^s. \]
Then, since $1-\frac{s}{q_0}=\frac{s}{p_0'}$, we have that 
\[  V_1^{\frac{1}{r_0}} V_2^{\frac{1}{r'_0}}
= \left(V_1^{\frac{1}{q_0}} V_2^{\frac{1}{p'_0}}\right)^s
 = \left(\overline{h}_2^{\frac{1}{q_0}} W^{\frac{s}{q_0}}  \overline{h}_1^{-\frac{s}{p_0'}} W^{\frac{s}{p_0'}}\right)^s
 = W_0^s. \]
By (C$_1'$),  $V_1 \in \A_{1,\B}$ and $[V_1]_{\A_{1,\B}} \leq c(d)K(B,W^{-s},r')$.  By (C$_1$), $V_2^{-1} \in \A_{1,\B}$ and so
$V_2\in \A_{\infty,\B}$ and $[V_2]_{\A_{\infty,\B}}= [V_2^{-1}]_{\A_{1,\B}} \leq c(d)K(\B,w^s,r)$.  Therefore, by Proposition~\ref{prop_8.7}, $W_0^s \in \A_{r_0,\B}$, and 
\[ [W_0^s]_{\A_{r_0,\B}} \leq c(d)[V_1]_{\A_{1,\B}}^{\frac{1}{r_0}} [V_2]_{\A_{\infty,\B}}^{\frac{1}{r_0'}}  
\leq c(d,p_0,q_0)K(B,W^{-s},r')^{\frac{1}{r_0}}K(\B,W^s,r)^{\frac{1}{r_0'}}. \]
Therefore, we have shown that \eqref{eqn:extrapol1} holds with constant
\begin{multline*}
    N_{p,q}(\B,p_0,q_0,p,q, [W^s]_{\A_{r,\B}}) \\
    =
C(d,p_0,q_0,p,q) N_{p_0,q_0}\big( C(d,p_0,q_0) K_{r'}([W^{-s}]_{\A_{r',\B}})^{\frac{1}{r_0}} 
K_r([W^s]_{\A_{r,\B}})^{\frac{1}{r_0'}}\big). 
\end{multline*}

\medskip

\textbf{Case II: $\mathbf{p_0=1,\, q_0<\infty}$}.  Fix $1<q_0<\infty$ and let $p,\,q,\,r$ be as in the statement of the theorem. Note that in this case, $s=q_0$.  Fix a matrix weight $W$ such that $W^{q_0}\in \A_{r,\B}$.  The proof is very similar to, but simpler than, the proof in Case I; we describe the changes.  To estimate the lefthand side of~\eqref{eqn:extrapol1}, we again apply duality, getting the function $h_2$ such that $ \| h_2 \|_{L^{r'}(\R^n)} = 1$.  Define $H_2$ and $\overline{h}_2$ as before, using the second iteration operator.  Then conditions (A$_1'$), (B$_1'$), and (C$_1'$) hold.  

By (B$_1'$)  and H\"older's inequality with $r$ and $r'$, we have
\begin{multline*}
    \int_{\R^n} |W(x) f(x)|^{q_0} h_2(x) \,dx \\
    \leq \int_{\R^n} |W(x) f(x)|^{q_0} \overline{h}_2(x) \,dx 
    \le \| f \|_{L^{q}(\R^n, W)}^{q_0} \| \overline{h}_2 \|_{L^{r'}(\R^n)}  2 \|f\|^{q_0}_{L^q(\R^n,W)}< \infty.
\end{multline*}
Define the matrix $W_0 = \overline{h}_2^{\frac{1}{q_0}}W$.  By (C$_1'$), $W_0^{q_0} \in \A_{1, \B}$.
Therefore, by our assumption~\eqref{eqn:extrapol0} and H\"older's inequality for $p$ and $p'$,
\begin{align*}
    \| f \|_{L^{q}(\R^n, W)} 
    & \leq \left( \int_{\R^n} |W(x)\overline{h}_2(x)^{\frac{1}{q_0}}f(x)|^{q_0} \,dx \right)^{\frac{1}{q_0}} \\
    & = \left( \int_{\R^n} |W_0(x) f(x)|^{q_0} \,dx \right)^{\frac{1}{q_0}} \\
    & \le N_{1, q_0}([W_0^{q_0}]_{\A_{1,\B}}) \left( \int_{\R^n} |W_0(x) g(x)| \,dx \right)\\
    & = N_{1, q_0}([W_0^{q_0}]_{\A_{1,\B}}) \left( \int_{\R^n} |W(x) g(x)| \overline{h}_2(x)^{\frac{1}{q_0}} \,dx \right)\\
    & \leq N_{1, q_0}([W_0^{q_0}]_{\A_{1,\B}}) \left( \int_{\R^n} |W(x)g(x)|^{p} \,dx \right)^{\frac{1}{p}} \left( \int_{\R^n} \overline{h}_2(x)^{\frac{p'}{q_0}} \,dx \right)^{\frac{1}{p'}}\\
    &  \leq 2^{\frac{1}{q_0}} N_{1, q_0}([W_0^{q_0}]_{\A_{1,\B}}) \| g \|_{L^p(\R^n, W)}.
\end{align*}
In the last inequality, we use the fact that $r'=\frac{p'}{q_0}$ since $ \frac{1}{p} - \frac{1}{q} = 1- \frac{1}{q_0} $.  As before, we can estimate $N_{1, q_0}([W_0^{q_0}]_{\A_{1,\B}})$ to get the desired expression for the constant.

\medskip

%%%%%%%%%%%%%%%%%%%%%%%%%%%%%%%%%%%%%%%%% infinity case
\textbf{Case III: $\mathbf{p_0>1,\, q_0=\infty}$}.  Fix $1<p_0<\infty$ and let $p,\,q,\,r$ be as in the statement of the theorem. Note that in this case, $s=p_0'$ and $r_0=\infty$.  Fix a matrix weight $W$ such that $W^{p_0'}\in \A_{r,\B}$.  Again, the proof is similar to, but simpler than, the proof in Case I.   Define $H_1$, $h_1$, and $\overline{h}_1$ as before, using the first iteration operator.  Then conditions (A$_1$), (B$_1$), and (C$_1$) hold.  

By the definition of $h_1$ and (A$_1$),  
\begin{equation} \label{eqn:fbound}
|W(x) f(x)| \overline{h}_1(x)^{-1} \leq |W(x) f(x)| h_1(x)^{-1} \leq \| f \|_{L^q(\mathbb{R}^n, W)}<\infty.
\end{equation}
Hence, if we define $W_0 = \overline{h}_1^{-1}W$, we have that 
\[
\| f \|_{L^\infty(\R^n, W_0)} \leq \| f \|_{L^q(\R^n, W)} < \infty.
\]
By (C$_1$), $\overline{h}^{p_0'}_1 W^{-p_0'} \in \A_{1,\B}$, so $W_0^{p_0'} \in \A_{\infty,\B}$.   Therefore, by our assumption \eqref{eqn:extrapol0},
\begin{align*}
    \| f \|_{L^q(\R^n, W)}  
     & = \left( \int_{\R^n} |W_0(x) f(x)|^q \overline{h}_1(x)^q \,dx \right)^{\frac{1}{q}} \\
    & \le \|f\|_{L^{\infty}(\R^n,W_0)}\|\overline{h}_1\|_{L^q(\R^n)} \\
    & \leq 2^{\frac{1}{p_0'}+1}N_{p_0, \infty}([W_0^{p_0'}]_{\A_{\infty,\B}}) \| g \|_{L^{p_0}(\R^n, W_0)} \\
    & \leq 2^{\frac{1}{p_0'}+1}N_{p_0, \infty}([W_0^{p_0'}]_{\A_{\infty,\B}})
    \bigg(\int_{\R^n} |\overline{h}_1(x)^{-1}W(x)g(x)|^{p_0}\,dx\bigg)^{\frac{1}{p_0}}.
\end{align*}
Arguing as we did to prove~\eqref{eqn:fbound}, we have that 
\[ |\overline{h}_1(x)^{-1}W(x)g(x)| \leq \overline{h}_1(x)^{\frac{q}{p}-1}\|g\|_{L^p(\R^n,W)}.  \]
Therefore, we can continue the above estimate to get
\begin{multline*} \bigg(\int_{\R^n} |\overline{h}_1(x)^{-1}W(x)g(x)|^{p_0}\,dx\bigg)^{\frac{1}{p_0}} 
\leq \|g\|_{L^p(\R^n,W)}\bigg(\int_{\R^n} \overline{h}_1(x)^{p_0(\frac{q}{p}-1)}\,dx\bigg)^{\frac{1}{p_0}}\\
= \|g\|_{L^p(\R^n,W)}\|\overline{h}_1\|_{L^q(\R^n)}^{\frac{q}{p_0}}
\leq 2^{(\frac{1}{p'_0}+1)\frac{q}{p_0}}\|g\|_{L^p(\R^n,W)}.  
\end{multline*}
If we combine these inequalities and estimate $[W_0^{p_0'}]_{\A_{\infty,\B}}$ as before, we get the desired bound.  This completes the proof.
\end{proof}

\section{The weak matrix Muckenhoupt property}
\label{section:weak-Muckenhoupt}

In this section we prove Theorem~\ref{thm:weak-muckenhoupt}.  For the convenience of the reader we recall the definitions we made in the Introduction.  Given a basis $\B$, we set
\[ \Omega = \bigcup_{B\in \B} B.  \]
We define an equivalence relation on $\B$ by saying that given $B,\,B'\in \B$, then $B\approx B'$ if there exists a finite sequence $\{B_i\}_{i=0}^k$, $B_i\in \B$, $B_0=B$, $B_k=B'$, and for $1\leq i \leq k$, $B_{i-1}\cap B_i \neq \emptyset$.  Since the sets $B_i$ are open, each intersection has positive measure.  Denote the equivalence classes by $\B_j$ and the union of sets in $\B_j$ by $\Omega_j$.  

Now let $W$ be a $d\times d$ matrix function defined on $\Omega$ such that the components of $W$ take values in $[-\infty,\infty]$.  (Without loss of generality, we may assume $W$ is zero on $\R^n\setminus \Omega$.)     We also assume that $W$ is self-adjoint and positive semi-definite, in the sense that $\langle W(x)v,v\rangle \geq 0$ for every $v\in \R^d$ and almost every $x\in \Omega$, using the conventions that $0\cdot \infty = 0$, and $s\cdot \infty+t\cdot\infty= \sgn(s+t)\infty$.

Let $f=(f_1,\ldots,f_d)$, where each $f_i$ is a measurable function taking values in $[-\infty,\infty]$. Define the product $W(x)f(x)$ again using the above conventions.   Let $L^p(\R^n,W)$ consist of all such $f$ that satisfy
\[ \int_{\R^n} |W(x)f(x)|^p\,dx < \infty.  \]
For such a matrix function $W$, define the  maximal operator $\wM_{W,\B}$ on $L^p(\R^n,W)$ by
\[  \wM_{W,\B}f(x) = \sup_{B\in \B} \avgint_B |W(x)f(y)|\,dy \cdot \chi_B(x).  \]
We assume that $\B$ has the weak matrix Muckenhoupt property for $W$ and $p$: that is, $\wM_{W,\B} : L^p(\R^n, W)\rightarrow L^p(\R^n)$. 

We will first prove that 
 $W$ is either trivial or non-trivial on each equivalence class $\Omega_j$.  
Fix $B\in \B$ and  suppose that $W$ is not nontrivial on $B$.  Then there exists $v \in \R^d$, $v \neq 0$, such that $|W(x)v|=0$ or $|W(x)v|=\infty$ on a set of positive measure.   Fix $S \subset B$, measurable and define $f = v \chi_S$.
Then for any $x \in B$,
\[
\widetilde{M}_{W, \mathcal{B}} f(x) \geq \avgint_B |W(x) v \chi_S(y)| \, dy = \frac{|S|}{|B|} |W(x) v|.
\]
Since $\widetilde{M}_{W, \mathcal{B}} : L^p(W) \to L^p(\R^n)$, we have that 
\begin{equation}\label{eq_*}
    \left( \frac{|S|}{|B|} \right)^p \int_B |W(x) v|^p \,dx \leq \|\widetilde{M}_{W, \mathcal{B}} f\|_{L^p(\R^n)}^p 
    \leq C\|f\|_{L^p(W)}^p    
    = C \int_S |W(x) v|^p \, dx.
\end{equation}

Let $S = \{x \in B : |W(x) v| = 0 \}$ and suppose $|S| > 0$. Then the righthand side of \eqref{eq_*} is $0$, so the lefthand side is also $0$. Thus,
\[
\int_B |W(x) v|^p \, dx = 0,
\]
and so $|W(x)v|=0$ a.e.~in $B$.  
Now suppose $|W(x)v| = \infty$ on a set of positive measure. For each $N>0$, let
\[
S = S_N = \{x \in B : |W(x)v| < N \}.
\]
Suppose that $S_N$ has positive measure. Then the lefthand side of the inequality \eqref{eq_*} equals $\infty$. So the righthand side of \eqref{eq_*} is $\infty$, which is a contradiction. Hence, none of the $S_N$ have positive measure, so
$|W(x)v| = \infty$ a.e.~in $B$.  

We can now show that $W$ is either trivial or nontrivial on each $\Omega_j$.  Fix $j$ and suppose that $W$ is trivial on some $B\in \B_j$.  Let $B'$ be any other set in $\B_j$ and let $\{B_i\}_{i=0}^N$ be the chain of overlapping sets in $\B_j$ connecting them.  
If $W$ is trivial on $B=B_0$, then it is trivial on the set $B_0\cap B_1$, which has positive measure, and so, by the argument above, is trivial on $B_1$.  Continuing in this way, by induction, we see that $W$ is trivial on $B'$.  Hence, $W$ is trivial on $\Omega_j$.  Therefore, if $W$ is nontrivial on any $B\in \B_j$, it must be nontrivial on $\Omega_j$.  

It is straightforward to show that $W$ is finite on $\Omega^*$. Let $W=\{w_{ij}\}_{i,j=1}^d$.   Fix $B\in \B^*$. Then for each basis vector $e_i$, $1\leq i\leq d$, the set $E_i = \{x \in B : |W(x)e_i| < \infty \}$
has full measure, and so $E = \cap_{i=1}^d E_i$ satisfies $|B \setminus E| = 0$. Fix $i$; then for every $x \in E$, $|W(x)e_i| < \infty$, and so the elements in the $i$-th row of $W(x)$ must be finite.  Hence,  $|w_{ij}(x)|<\infty$ a.e.~ for all $1\leq i,\,j\leq d$.

It is somewhat more complicated to show that $W$ is invertible almost everywhere on $\Omega^*$.  Fix a set $B\in \B^*$.  
 Since $W$ is positive semi-definite and measurable, by Lemma~\ref{lemma:diagonal} there exists a measurable orthogonal matrix $U$ such that
\[
U(x) W(x) U^t(x) = \diag(\lambda_1(x), \dots, \lambda_d(x)).
\]
Hence, the eigenvalues of $W$ are measurable functions, and $0\leq \lambda_i(x)<\infty$ for every $i$ and a.e.~$x\in B$.  Define
\[
 \lambda_{\min}(x) = \min\{\lambda_1(x), \dots, \lambda_d(x)\}.
\]
We claim that $\lambda_{\min}$ is also a measurable function.   To show this, define  $H : B \to \{1, \dots, d\}$ by
\[
H(x) = \min \{i \in \{1, \dots, d\} : \lambda_i(x) = \lambda_{\min}(x)\}.
\]
Then $H$ is  measurable, and so is
\[
\lambda_{\min}(x) = \lambda_{H(x)}(x).
\]
(See~\cite[Section~10.2]{domelevo2024boundedness}.)  

Denote the columns of $U$ by ${u}_1(x), \dots, {u}_d(x)$; these are measurable vector-valued functions. If we define
\[
{v}_{\min}(x) = {u}_{H(x)}(x),
\]
then ${v}_{\min}$ is also measurable. To see this, for each $1\leq i \leq d$, let $B_i=\{x\in B : H(x)=i\}$.  Since $H$ is a measurable function, each $B_i$ is measurable, and so
\[ {v}_{\min}(x) = \sum_{i=1}^d {u}_i(x)\chi_{B_i}(x) \]
is a measurable function as well.

For almost every $x$, ${v}_{\min}(x)$ is the unit eigenvector of $W(x)$ associated with $\lambda_{\min}(x)$,
We now repeat the above argument, replacing the constant vector $v$ by the vector function ${v}_{\min}$. Let
$S = \{x \in B : \lambda_{\min}(x) = 0\}$ and define
$f(x) = {v}_{\min}(x) \chi_S(x)$.
Then for all $x \in B$,
\[
\widetilde{M}_{W, B} f(x) \geq \avgint_B |W(x) {v}_{\min}(y)| \, dy.
\]
Since $\widetilde{M}_{W, B} : L^p(W) \to L^p(\R^n)$, we have that
\[
\int_B \left(  \avgint_B |W(x) {v}_{\min}(y)| \, dy \right)^p \, dx \leq C \int_S |W(x) {v}_{\min}(x)|^p \, dx = 0.
\]
Hence, by H\"older's inequality and Fubini's theorem,
\[
\int_S \int_B |W(x) {v}_{\min}(y)| \, dy \, dx = 0.
\]
Therefore, for almost every $y \in S$, 
\[
\int_B |W(x) v_{\min}(y)| \, dx = 0,
\]
so, if we fix such a $y$, for almost every $x \in B$,
\[
|W(x) {v}_{\min}(y)| = 0.
\]
But  $W$ is nontrivial on $B$, so this is impossible unless $|S| = 0$. Therefore, $\lambda_{\min}(x) > 0$ almost everywhere on $B$, and so $W^{-1}(x)$ exists almost everywhere on $B$.  Thus we have shown that $W$ is finite and invertible on $\Omega^*$.

To show that
  \begin{equation} \label{eqn:B*-Ap-cond}
\sup_{B\in \B^*} \bigg(\avgint_B \bigg(\avgint_B |W(x)W^{-1}(y)|_{\op}^{p'}\,dy\bigg)^{\frac{p}{p'}}\,dx\bigg)^{\frac{1}{p}} < \infty, 
\end{equation}
it is enough to note that since $W$ is finite and invertible on $\Omega^*$, we have by a change of variables that $M_{W,\B^*}$ is bounded on $L^p(\R^n)$.  Hence, we have that for all $B\in \B^*$, the averaging operators
\[ A_Bf(x) = \avgint_B f(y)\,dy \cdot \chi_B(x) \]
are uniformly bounded.  This in turn implies that \eqref{eqn:B*-Ap-cond} holds; the proof is essentially identical to the proof for a basis of cubes, but it requires multiple steps, which we sketch and give references for.   The uniform boundedness of averaging operators is equivalent to the reducing matrix condition~\eqref{eqn:reducing-op1} given in Proposition~\ref{prop:reducing-op}.  For a proof for a basis of cubes, see~\cite[Theorem~1.18]{MR3803292}; also see~\cite[Proposition~2.1]{MR2015733}.  Then, the proof that the reducing operator condition is equivalent to the $\A_{p,\B^*}$ condition is again the same as for a basis of cubes:  see Roudenko~\cite[Lemma~1.3]{MR1928089}.  Note that by Lemma~\ref{SelfAdjointCommutes} applied to the reducing operator condition, we have that if $W \in \A_{p,\B}$ condition, then $W^{-1} \in \A_{p',\B^*}$. 

Finally, we will use~\eqref{eqn:B*-Ap-cond} to show that $W^{-1}\in L^{p'}_\loc(\Omega^*)$; the proof that $W \in L^{p}_\loc(\Omega)$ follows by applying the same argument to the $\A_{p',\B^*}$ condition.  Fix $B\in \B^*$; it follows from~\eqref{eqn:B*-Ap-cond} that for almost every $x\in B$, 
\[ \int_B |W^{-1}(y)W(x)|_{\op}^{p'}\,dy < \infty.  \]
Since the set where $W(x)$ is finite and invertible also has full measure, we can fix such an $x$ for which this integral is finite.  Let $0<\lambda_1 \leq \lambda_2 \leq \cdots \leq \lambda_d<\infty $,  be the eigenvalues of $W(x)$, with associated unit eigenvectors $v_i$.  But then we have that
\begin{multline*}
    \int_B |W^{-1}(y)|_{\op}^{p'}\,dy
    \approx \lambda_1^{-p'}\sum_{i=1}^d \int_B \lambda_i^{p'} |W^{-1}(y)v_i |^{p'}\,dy \\
    = \lambda_1^{-p'} \sum_{i=1}^d \int_B |W^{-1}(y)W(x)v_i |^{p'}\,dy 
    \approx \int_B |W^{-1}(y)W(x)|_{\op}^{p'}\,dy < \infty.  
\end{multline*} 
Since every compact subset of $\Omega^*$ is contained in a finite union of sets $B\in \B^*$, it follows that $W^{-1}\in L^{p'}_\loc(\Omega^*)$.
This completes the proof of Theorem~\ref{thm:weak-muckenhoupt}.

\begin{remark}  \label{remark:trivial-matrix}
    We consider briefly how to apply Theorem~\ref{thm:weak-muckenhoupt} in practice.  Given a basis $\B$, suppose that there exists a single equivalence class of sets in it.  If $W$ is nontrivial then we can proceed as before, extending $W$ to all of $\R^n$.  However, if $W$ is trivial, there are two cases.  If there exist vectors $v\in \R^d$ such that $|W(x)v|=0$ a.e., then they form a subspace and if we restrict to the subspace orthogonal to it, there is a matrix $V$ that is nontrivial and acts on this subspace as $W$.  Therefore, by a change of variables we can assume that $W$ is defined on $\R^c$ for $1\leq c<d$.  This approach was first noted in passing by Treil and Volberg~\cite[Section~2]{MR1428818}.  

    If there exists a vector such that $|W(x)v|=\infty$ a.e., we cannot argue the same way since these vectors do not form a subspace.  We could, nevertheless, take their span  as a subspace and consider its orthogonal complement.  Alternatively, we could consider all the $j$ such that there exists $i$ with $|w_{ij}(x)|=\infty$ on a set of positive measure, and take the subspace orthogonal to the span of the basis vectors $e_j$.  (Equivalently, we are considering vector functions that are $0$ in their $j$-th coordinate.)

    Finally, if there is more than one equivalence class, these arguments could be repeated on each one, but this might result in $W$ being of a different dimension on each equivalence class.   Given that this may have limited utility, we leave this to the interested reader.
    \end{remark}

\section{Multiparameter $\mathcal{A}_{p}$ matrix weights}
\label{section:multiparameter}

In this section we show that the multiparameter basis $\Rdf^\alpha$ is a matrix Muckenhoupt basis and prove Theorem~\ref{theo_boundednes}.  Our approach  is to follow the scalar case, where the maximal operator $M_{\mathcal{R}^\alpha}$ is shown to be dominated by iterates of the Hardy-Littlewood maximal operator acting on each $\R^{a_k}$. Our proof generalizes that given by Vuorinen~\cite[Theorem~1.3]{VUORINEN2024109847} for the biparameter case; while similar, there are a number of details to check and so we include the full proof. 

Given $\alpha=(a_1,\ldots,a_j)$ such that $a_1+\cdots+a_j=n$, we write $\R^n=\R^{a_1}\times \R^{a_2}\times \dots \times\R^{a_j}$.  Then set $\Rdf^\alpha$ is the collection of all rectangles $R=Q_1\times Q_2 \times \dots \times Q_j$ where for each $k$, $1\leq k \leq j$, $Q_k \subset \R^{a_k}$ is a cube.   To express iterates as simply as possible, and so simplify our argument, we introduce some notation.  Given  ${x} \in \R^n = \R^{a_1}\times \dots \times \R^{a_j}$, we will denote it by ${x}=(x_1,\dots,x_j)$,   $x_k \in \R^{a_k}$.  We will define the vector $\hat{x}_k$ in $\R^{n-\alpha_k}$ by  
\[ \hat{x}_k =(x_1,\ldots,x_{k-1},x_{k+1},\ldots,x_j). \]
For iterated integrals, we will write $d\hat{x}_k=dx_1\dots \,dx_{k-1}dx_{k+1} \dots \,dx_j$ and $d{x}=dx_1 \dots \,dx_j= dx_k d\hat{x}_k $.  Finally, given $R = Q_1 \times Q_2 \times \dots \times Q_j$, for each $1\leq k \leq j$, we define  $\hat{R}_k = Q_1\times \dots \times Q_{k-1}\times Q_{k+1} \times \dots \times Q_j$, and $\hat{\R}^n_k = \R^{a_1}\times \cdots \times \R^{a_{k-1}} \times \R^{a_{k+1}}\times \cdots \times \R^{a_{j}}$.

If the function $F$ is defined on $\hat{\R}^n_k$, we will write $F(\hat{x}_k)=F(x_1,\ldots,x_{k-1},x_{k+1},\ldots,x_j)$. Similarly, if $G$ is defined on $\R^n$, we will write $G({x})=G(x_1,\dots, x_j)=G(x_k,\hat{x}_k)$.
In particular, if $W$ ia a matrix weight, then define
\[
W_{\hat{x}_k} (x_k) = W(\hat{x}_k, x_k) = W({x}).
\]
If we fix $\hat{x}_k$ and allow $x_k \in \R^{a_k}$ to vary, we can regard $W_{\hat{x}_k} (x_k)$ as a matrix weight on $\R^{a_k}$.

For the proof we need several  lemmas.  The first was proved in the biparameter version in~\cite[Lemma 3.6]{domelevo2024boundedness}.  Our proof is an adaptation of theirs; however, there are several important differences in the details, and so we include the proof.

\begin{lemma}\label{onepMatrixtoMulti}
    Given $1<p<\infty$,  let $W \in \A_{p,\mathcal{R}^{\alpha}}$. Then, for every $k$, $1\leq k \leq j$, and for almost every $\hat{x}_k \in \hat{\R}^n_k$,  
$W_{\hat{x}_k}(\cdot)$ is a one-parameter matrix weight defined on $\R^{a_k}$ that satisfies
\[
[W_{\hat{x}_k}(\cdot)]_{\A_{p,\Q}(\R^{a_k})} \lesssim [W]_{\A_{p,\mathcal{R}^{\alpha}}}.
\]
\end{lemma} 

\begin{proof}
To prove this result it will suffice to show that for almost every $\hat{x}_k \in \hat{\R}^n_k$,
\begin{equation}  \label{eqn:OMM1}
\sup_{\Q^0_k} \bigg(\avgint_Q \bigg( \avgint_Q |W_{\hat{x}^k}(x_k) W_{\hat{x}^k}^{-1}(y_k)|_{\op}^{p'}\,dy_k
\bigg)^{\frac{p}{p'}}dx_k \bigg)^{\frac{1}{p}} \leq C[W]_{\A_p, \Rdf^\alpha}, 
\end{equation}
where the supremum is taken over the set $\Q^0_k$  of cubes in $\R^{a_k}$ whose vertices have rational coordinates. (The full condition follows at once since every cube in $\R^{a_k}$ is contained in a cube in $\Q_k^0$ whose sidelength is comparable.    We will prove~\eqref{eqn:OMM1} by proving an equivalent statement for reducing operators.  

 Fix a cube $Q_k \subset \R^{a_k}$, $1\leq k \leq j$.  By Theorem~\ref{thm:weak-muckenhoupt}, $W\in L^p_\loc(\R^n)$, and so, given any $R\in \Rdf^\alpha$ whose $k$-th term is $Q_k$, we have that
 \[ \int_{\hat{R}_k} \int_{Q_k} |W_{\hat{x}_k}(x_k)|_{\op}^p \,dx_k d\hat{x}_k = \int_R |W(x)|_{\op}^p \,dx < \infty. \]
 Hence, for a.e.~$\hat{x}_k\in \hat{R}_k$, $W_{\hat{x}_k}(\cdot)$ is finite for almost every $x_k \in Q_k$, and in fact in $L^p(Q_k)$.  Similarly, since $W^{-1} \in L^{p'}_\loc(\R^n)$, $W_{\hat{x}_k}^{-1}(\cdot)\in L^{p'}(Q_k)$.   Therefore, we can form the reducing operators   $\mathcal{W}^p_{\hat{x}_k, Q_k}$ and $\overline{W}_{\hat{x}_k,Q_k}^{p'}$.  Moreover, we can do so in such a way that they are measurable functions in $\hat{x}_k$.  (See~\cite[Proposition~A.8]{domelevo2024boundedness}.) 

By Proposition~\ref{prop:reducing-op} applied to the basis $\Q^0_k$, to prove~\eqref{eqn:OMM1} it will suffice to show that for every $Q_k\in \Q^0_k$ and a.e.~$\hat{x}_k \in \hat{\R}^n_k$, 
\begin{equation} \label{eqn:OMM2}
    |W_{\hat{x}_k, Q_k}^p \overline{W}_{\hat{x}_k,Q_k}^{p'}|_{\op} \leq C[W]_{\A_{p,\Rdf^\alpha}}. 
\end{equation}

To prove this we will follow the argument in~\cite[Lemma~3.6]{domelevo2024boundedness}.  However, to adapt their argument, we will need to apply the Lebesgue differentiation theorem with respect to the basis of rectangles. By the theorem of Jessen, Marcinkiewicz and Zygmund (and the counter-example of Saks), we need to show that the functions we are differentiating are in $L^s_{\loc}$ for some $s>1$.  (See de Guzm\'an~\cite{MR0457661}.)  To do this, we argue as follows.  Fix a cube $Q_k \in \Q_k^0$.  Since $\A_{p,\Rdf^\alpha} \subset \A_{p,\Q}$, we have that $W\in \A_{p,\Q}$.  Therefore, by a result due to the first author and Penrod~\cite[Theorem~1.5]{DCU-MP-2024}, we have that there exists $s>1$ such that $W\in \A_{sp,\Q}$ and $W^{-1}\in \A_{sp',\Q}$.  Therefore, again by Theorem~\ref{thm:weak-muckenhoupt}, we have that $W\in L^{sp}_\loc(\R^n)$ and $W^{-1}\in L^{sp'}_\loc(\R^n)$.  

Now for every $v\in \mathbb{Q}^{d}$, define the functions
\[ F_{Q_k,v}(\hat{x}_k)  = \int_{Q_k} |W(x_k,\hat{x}_k) v|^p\, dx_k, 
\quad
G_{Q_k,v}(\hat{x}_k)  = \int_{Q_k} |W^{-1}(x_k,\hat{x}_k) v|^{p'}\, dx_k.  \]
Then by Fubini's theorem and H\"older's inequality, we have that $F_{Q_k,v}, \, G_{Q_k,v} \in L^s_\loc(\hat{\R}^n_k)$.  Let $A_{Q_k,v}$ be the set of strong Lebesgue points of $F_{Q_k,v}$ in $\hat{\R}^n_k$ and let $B_{Q_k,v}$ be  the strong Lebesgue points of $G_{Q_k,v}$ in $\hat{\R}^n_k$.  Then, since the intersection is countable, the set of points 
\[ C_{k} = \bigcap_{Q_k \in \Q^0_k}\bigcap_{v\in \mathbb{Q}^{d}} \big(A_{Q_k,v} \cap B_{Q_k,v}\big) \]
is such that $|\hat{\R}^n_k\setminus C_{k}|=0$.  Fix $\hat{x}_k \in C_{k}$ and let $R^l=P^l_1\times \cdots \times P^l_k\in \Rdf^\alpha$ be such that $P^l_k=Q_k$ and  the sequence $\{\hat{R}^l_k\}_{l=1}^k$ shrinks to $\hat{x}_k$ as $l\to \infty$.  Then by the Jessen-Marcinkiewicz-Zygmund theorem, we have that for every $v\in \mathbb{Q}^{d}$,
\[ \lim_{l\to \infty} \avgint_{\hat{R}^l_k} \avgint_{Q_k} |W(x_k,\hat{y}_k)v|^p\,dx_k d\hat{y}_k 
= \avgint_{Q_k} |W(x_k,\hat{x}_k)v|^p\,dx_k \approx |W_{\hat{x}_k,Q_k}^pv|^p. 
\]
Since the sequence on the lefthand side converges, it is bounded; in particular, if  $\{e_i\}_{i=1}^{d}$ is the standard orthonormal basis in $\R^{d}$, the sequence $\{W_{R^l}\}_{l=1}^\infty$ of reducing operators of $W$ on $R^l$ satisfies
\[ |W_{R^l}^p|_{\op}^p \approx \sum_{i=1}^d |W_{R^l}e_i|^p 
\approx \sum_{i=1}^d   \avgint_{\hat{R}^l_k} \avgint_{Q_k} |W(x_k,\hat{y}_k)e_i|^p\,dx_k d\hat{y}_k, \]
and so is uniformly bounded.  Therefore, by passing to a subsequence, we may assume that the sequence  converges to some matrix $\W$ in matrix operator norm.  Passing to the limit, we have that for any $v\in \mathbb{Q}^d$, 
\[ |W_{\hat{x}_k,Q_k}^pv| \approx |\W v|. \]
By density, the same must hold for every $v\in \R^d$. 

Now repeat this argument, replacing $W$ by $W^{-1}$ and $p$ by $p'$. This yields a matrix $\overline{\W}$ such that for every $v\in \R^d$, 
\[ |\overline{W}_{\hat{x}_k,Q_k}^{p'}v| \approx |\overline{\W} v|. \]
Therefore, by inequality~\eqref{opNorm-twomatrices} and Proposition~\ref{prop:reducing-op},
\[ |W_{\hat{x}_k, Q_k}^p \overline{W}_{\hat{x}_k,Q_k}^{p'}|_{\op} \lesssim |\W \overline{\W}|_{\op}
= \lim_{l\to \infty} |W_{R^l}^p \overline{W}_{R^l}^{p'}|_{\op} \leq C[W]_{\A_{p,\Rdf^\alpha}}. 
\]
Since this is true for every $Q_k \in \Q_k^0$ and almost every $\hat{x}_k$, this proves~\eqref{eqn:OMM2} and so completes the proof.
\end{proof}

\medskip

We now consider iterations of the convex set-valued maximal operator, which will play a role in our proof. From Definition~\ref{defn:convex-set-max-op}, if $F$ is a locally integrably bounded function, let
\[
M_{\Rdf^\alpha} F(x) = \clconv \bigg( \bigcup_{R \in \mathcal{R}^{\alpha}}\avgint_R F(y)\,dy \cdot\chi_R(x) \bigg).
\]
We also define the  one-parameter convex set-valued maximal operators. Given $F : \R^n \to \mathcal{K}_{bcs}$ be locally integrably bounded.   Then for almost every $\hat{x}_k$, $F_{\hat{x}_k} : \R^{a_k} \to \mathcal{K}_{bcs}(\R^d)$ is also locally integrably bounded, so for $1\leq k\leq j$ we can define
\[
M^k F_{\hat{x}_k}(x_k) = M\left(F(\hat{x}_k,\cdot)\right)(x_k),
\]
where on the righthand side, $M$ denotes the convex set-valued maximal operator defined with respect to cubes on $\R^{a_k}$.
We now prove the analog of iterations of the scalar maximal operator  for convex set-valued maximal operators. The following result was proved in the biparameter case in~\cite[Lemma~3.1]{VUORINEN2024109847} and we follow his proof.

\begin{lemma}\label{lemma-M^aSubset}
    Let \( F: \R^{n} \to \mathcal{K}_{bcs}(\R^d) \) be locally integrable, bounded, and assume further that for every $k$, $1\leq k \leq j$, \( M^k M^{k+1} \ldots M^j F \) is locally integrably bounded. Then
\[
M_{\Rdf^\alpha} F \subseteq M^1 M^2 \cdots M^j F.
\]
\end{lemma} 

\begin{proof}
    Fix  \( x \in \R^{a_1} \times \cdots \times \R^{a_j} \) and a multi-parameter rectangle \( R = Q_1 \times \cdots \times Q_j \) that contains~$x$.   If $f$ is a selection function on $Q_j$, then 
\[
\int_{Q_j} f(y_1, \ldots, y_j)  \,dy_j \in \int_{Q_j} F(y_1, \ldots, y_j)  \,dy_j \subseteq M^j F(y_1, \ldots, y_{j-1},x_j).
\]
Since $M^j F$ is locally integrably bounded, if we  integrate  over \( Q_{j-1} \), we get
\begin{multline*}
    \avgint_{Q_{j-1} \times Q_j} f  \,dy_{j-1}  \,dy_j = \avgint_{Q_{j-1}} \left( \avgint_{Q_j} f(y_1, \ldots, y_{j-1}, y_j)  \,dy_j \right) \,dy_{j-1} \\
    \in \avgint_{Q_{j-1}} M^j F(y_1, \ldots, y_{j-1},x_j)  \,dy_{j-1} \subseteq M^{j-1} M^j F(y_1,\dots, x_{j-1},x_j).
\end{multline*}

We can repeat this argument inductively, integrating over $Q_{j-2},\dots ,Q_1$ in turn.  This yields
\[
\avgint_R f \,dy \in M^1 M^2 \cdots M^j F(x).
\]
Since this holds for every selection function \( f \), we have that
\[
\avgint_R F \,dy \subset M^1 M^2 \cdots M^j F(x), 
\]
and so by the definition of the convex set-valued maximal operator, since the righthand side is convex,
\[
M_{\Rdf^\alpha} F(x)  \subset M^1 \cdots M^j F(x).
\]
\end{proof}

\begin{theorem}\label{thm.convMaxBoundedness}
   Fix \( p \in (1, \infty] \) and \( W \in \A_{p, \mathcal{R}^\alpha} \). If $F \in L^p_{\K}(W)$, then 
\begin{equation} \label{eqn:CMB1}
\|M_{\Rdf^\alpha} F\|_{L^p_{\K}(W)} \leq C(n, d, p) [W]_{\A_{p, \mathcal{R}^{\alpha}}}^{jp'} \|F\|_{L^p_{\K}(W)}.
\end{equation}
If \( p = \infty \), \eqref{eqn:CMB1} holds with constant \( [W]_{\A_{\infty}, \mathcal{R}^\alpha} \) instead of \( [W]_{\A_{\infty, \mathcal{R}^\alpha}}^{j} \).
\end{theorem} 

To prove this result we need the following lemma, which extends a well-known result in the scalar case (see, for instance, \cite[Lemma~3.30]{MR3026953}) to convex set-valued functions.

\begin{lemma} \label{lemma:montone-max-op}
    Let $F: \R^{n} \to \mathcal{K}_{bcs}(\R^d)$ be a locally integrably bounded function and define the sequence $\{F_k\}_{k=1}^\infty$ by
    \[ F_k(x) = \big(F(x) \cap k\overline{\mathbf B}\big)\cdot \chi_{B(0,k)}(x). \]  
    Then for any matrix weight $W$, and almost every $x\in \R^d$,
    \[ |W(x)M_{\Rdf^\alpha} F(x)| = \lim_{k\to \infty} |W(x)M_{\Rdf^\alpha} F_k(x)|. \]
\end{lemma}

\begin{proof}
    Since for all $k$, $F_k(x) \subset F_{k+1}(x) \subset F(x)$, by Lemma~\ref{lemma_5.5}, $M^\alpha_\K F_k(x) \subset M^\alpha_\K F_{k+1}(x) \subset M^\alpha_\K F(x) $.  Hence, it is immediate that 
    \[ \lim_{k\to \infty} |W(x)M_{\Rdf^\alpha} F_k(x)| \leq |W(x)M_{\Rdf^\alpha} F(x)|. \]

    To prove the reverse inequality first note that arguing as in~\cite[Lemma~5.6]{bownik-cruz-uribe}, if we define
    \[ \widehat{M}_{\Rdf^\alpha} F(x) = \cl\bigg(\bigcup_{R\in \Rdf^\alpha} \avgint_R F(y)\,dy \cdot \chi_{R}(x)\bigg), \]
    where $\cl(E)$ denotes the closure of $E$,   then $|W(x)\widehat{M}_{\Rdf^\alpha} F(x)|= |W(x){M}_{\Rdf^\alpha} F(x)|$.  Fix a set $R\in \Rdf^\alpha$ and let $f\in S^1(F,R)$ be any selection function.  Define the sequence $\{f_k\}_{k=1}^\infty$ by
    \[ f_k(x) = \frac{f(x)}{|f(x)|} \min(|f(x)|,k) \cdot \chi_{B(0,k)}(x). \]  
    Then $f_k \to f$ pointwise a.e.  Moreover, since $F$ is locally integrably bounded, there exists $K\in L^1(R)$ such that $|f_k(x)|\leq |f(x)|\leq K(x)$ for $x\in R$.    Therefore, by the dominated convergence theorem,
    \[ \lim_{k\to \infty} \avgint_R f_k(y) \,dy = \avgint_R f(y)\,dy.  \]

    We can now estimate as follows.  Fix $\epsilon>0$ and $x\in \R^d$.  By the definition of $\widehat{M}_{\Rdf^\alpha}$ and the definition of the Aumann integral, there exists $R\in \Rdf^\alpha$ containing $x$ and a selection function $f\in S^1(F,R)$ such that 
    \begin{multline*} 
    |W(x) {M}_{\Rdf^\alpha}F(x)| 
     = |W(x) \widehat{M}_{\Rdf^\alpha}F(x)| \\
     \leq (1+\epsilon) \bigg|W(x) \avgint_R F(y)\,dy\bigg| 
     \leq (1+\epsilon)^2 \bigg|W(x) \avgint_R f(y)\,dy\bigg|; 
    \end{multline*}
    by the continuity of the norm $\rho(v)=|W(x)v|$, the righthand term equals
    \begin{multline*}
    (1+\epsilon)^2\lim_{k\to \infty} \bigg|W(x) \avgint_R f_k(y)\,dy\bigg|  \\
     \leq (1+\epsilon)^2\lim_{k\to \infty} \bigg|W(x) \avgint_R F_k(y)\,dy\bigg| 
     \leq (1+\epsilon)^2\lim_{k\to \infty} |W(x)M_{\Rdf^\alpha} F_k(x)|. 
    \end{multline*}
    Since $\epsilon$ is arbitrary, we get the desired inequality.
\end{proof}

\begin{proof}[Proof of Theorem~\ref{thm.convMaxBoundedness}]
    Suppose first that $1<p<\infty$.  By Lemma~\ref{lemma:montone-max-op}, it will suffice to prove~\eqref{eqn:CMB1} for $F \in L^\infty_\K(\R^n,\R^d) \cap L^p_{\K}(\R^n,W)$.   The convex set-valued maximal operator is bounded on $L^\infty_{\K}$; this was proved in~\cite[Proposition~5.2]{bownik-cruz-uribe} for the basis of cubes, but the same proof works for any basis.  Therefore, we have that $M_{\Rdf^\alpha}F$ and all the iterates $M^kM^{k+1}\cdots M^j F$ are bounded, and so locally integrably bounded.

For all $k$, if $W\in \A_{p,\Q}(\R^{a_k})$ and $1<p \leq \infty$, by~\cite[Theorem 6.9]{bownik-cruz-uribe} we have  that 
\begin{equation} \label{eq.boundednessM^a}
\| M^k F \|_{L_\K^p(\R^{a_k}, W)} \leq C [W]_{\A_{p,\Q}}^{p'} \| F \|_{L_\K^p(\R^{a_k}, W)}.
\end{equation}
    We can now estimate as follows.  By  Lemma \ref{lemma-M^aSubset} and by Fubini's theorem and Lemma \ref{onepMatrixtoMulti} applied repeatedly,
\begin{align*}
\int_{\R^n} | W(x) M_{\Rdf^\alpha} F(x) |^p \,dx 
& \leq \int_{\hat{\R}^n_1} \int_{\R^{a_1}} | W(x) M_{\K}^1 \cdots M_{\K}^j F(x) |^p \,dx_1 d\hat{x}_1 \\
& \leq \int_{\hat{\R}^n_1} C\left[ W_{\hat{x}_1}(\cdot) \right]_{\A_{p,\Q}(\R^{a_1})}^{pp'}| \int_{\R^{a_1}} 
W(x) M_{\K}^2 \cdots M_{\K}^j F(x) |^p \,dx_1 d\hat{x}_1 \\
&  \leq C\left[ W \right]_{\A_{p,\Rdf^\alpha}}^{pp'} 
 \int_{\hat{\R}^n_2} \int_{\R^{a_2}} |W({x}) M_{\mathcal{K}}^2 \cdots M_{\mathcal{K}}^j F({x})|^p \, dx_2 d\hat{x}_2 \\
& \leq C\left[ W \right]_{\A_{p,\Rdf^\alpha}}^{2pp'}
 \int_{\hat{\R}^n_3} \int_{\R^{a_3}} |W({x}) M_{\mathcal{K}}^3 \cdots M_{\mathcal{K}}^j F({x})|^p \, dx_3 d\hat{x}_3 \\
 & \qquad \vdots \\
 & \leq C\left[ W \right]_{\A_{p,\Rdf^\alpha}}^{jpp'}
 \int_{\R^n}  |W({x})  F({x})|^p \, d{x}. 
\end{align*}
This proves~\eqref{eqn:CMB1} when $p$ is finite.

For \( p = \infty \), assume that \( F \in L_\K^\infty(W) \). We will prove the desired estimate using the definition of \( \A_{\infty, \mathcal{R}^{\alpha}} \);  our argument is similar to the \( L^\infty \)-estimate of \( M_W \)  in~\cite[Proposition 6.8]{bownik-cruz-uribe}, but following the argument in~\cite[Theorem~3.2]{VUORINEN2024109847} we fill a lacuna in that proof.  As before, the  function \( f \) is locally integrably bounded. Let \( \Rdf^{\alpha}_0 \) denote the set of all multiparameter rectangles in \( \R^{a_1} \times \cdots \times \R^{a_j} \) whose vertices have rational coordinates.  Then, if we argue as in~\cite[Proposition~5.7]{bownik-cruz-uribe}, we have that 
    \[ \widehat{M}_{\Rdf^\alpha} F(x) = \cl\bigg(\bigcup_{\Rdf\in \Rdf^\alpha_0} \avgint_R F(y)\,dy \cdot \chi_{R}(x)\bigg). \]
For every \( R \in \Rdf^\alpha_0 \) and almost every $x\in R$,
\[
\avgint_R | W(x) W^{-1}(y) |_{\op}  \,dy \leq [W]_{A_{\infty, R^{\alpha}}}.
\]
Since $\Rdf^\alpha_0$ is countable, by a standard argument this inequality holds for almost every $x$ and every $R$ containing $x$.   Fix such an $x$ and $R$ containing it; then we have that 

\begin{multline*}
   \left|W(x)  \avgint_R  F(y)  \,dy \right|   = \left| \avgint_R W(x) W(y)^{-1} W(y) F(y)  \,dy \right| \\
   \leq \avgint_R | W(x) W^{-1}(y) |_{\op}  \,dy \|WF\|_{L^{\infty}(\R^n;\R^d)} \le [W]_{\A_{\infty, R^{\alpha}}} \| F \|_{L^\infty_\K(W)}. 
\end{multline*}
Hence, for almost every $x$,
\[
    | W(x) M_{\Rdf^\alpha} F(x) | = |W(x)\widehat{M}_{\Rdf^\alpha}F(x)|
     \leq [W]_{\A_{\infty, R^{\alpha}}} \| F \|_{L^\infty_\K(W)}.
\]
The desired estimate follows at once.
\end{proof}

We can now prove our main result.   

\begin{proof}[Proof of Theorem~\ref{theo_boundednes}]
We will prove this assuming $1\leq p<\infty$;   the proof when $p=\infty$ is almost identical with the obvious changes and so is omitted.  
Fix $W\in \A_{p,\Rdf^\alpha}$.  Let \( f \in L^p(\R^n; \R^d) \) and define  $F(x) = \clconv \left(\{-f(x), f(x)\}\right)$.  Then $F\in L^p_\K(\R^n)$; moreover, we have that $W^{-1}F$ is locally integrably bounded.  To see this note that for any $R\in \Rdf^\alpha$, by H\"older's inequality,
\[ \int_R |W^{-1}(y)F(y)|\,dy 
\leq |W^{-1}(x)|_{\op} \bigg( \int_R |W(x)W^{-1}(y)|_{\op}^{p'}\,dy \bigg)^{\frac{1}{p'}}
\|F\|_{L^p_\K(\R^n)}. 
\]
Since $W\in \A_{p,\Rdf^\alpha}$, the first two terms on the righthand side are finite for almost every $x\in R$.
Therefore, by~\cite[Lemma~3.4]{VUORINEN2024109847},
\[
M_{W,\Rdf^\alpha} f(x) \lesssim |W(x) M_{\Rdf^\alpha} (W^{-1} F)|.
\]
(Note that this result is stated and proved for the biparamenter basis, but the same proof holds for any basis $\Rdf^\alpha$.)
Hence, by Theorem \ref{thm.convMaxBoundedness},
\begin{multline*}
    \|M_{W,\Rdf^\alpha} f\|_{L^p(\R^n)} 
    \lesssim \|W M_{\Rdf^\alpha} (W^{-1} F)\|_{L^p(\R^n)} \\
    \lesssim [W]_{\A_{p, \mathcal{R}^{\alpha}}}^{jp'} \|W^{-1} F\|_{L^p_\K(W)} 
    = [W]_{\A_{p, \mathcal{R}^{\alpha}}}^{jp'} \|F\|_{L^p_\K(\R^n)}
    =[W]_{\A_{p, \mathcal{R}^{\alpha}}}^{jp'} \|f\|_{L^p(\R^n,\R^d)}.
\end{multline*}
\end{proof}

\bibliographystyle{plain}
\bibliography{matrix-extrapol-bases}

\end{document}